\documentclass[12pt, a4paper, reqno]{amsart}
\usepackage{amssymb}
\usepackage{amsbsy}
\usepackage{xcolor}
\usepackage[normalem]{ulem}
\usepackage[hdivide={2.5cm,,2.5cm}, vdivide={3.5cm,,2.8cm}]{geometry}
\usepackage{amssymb}
\usepackage{amsthm}
\usepackage{amsmath}
\usepackage{mathtools}
\usepackage[pctexwin]{graphicx}

\usepackage{pdfsync}

\makeatletter
\newcommand*{\Log}{\mathop{\operator@font Log}\nolimits}
\newcommand*{\Arg}{\mathop{\operator@font Arg}\nolimits}
\newcommand*{\tg}{\mathop{\operator@font tg}\nolimits}
\newcommand*{\ctg}{\mathop{\operator@font ctg}\nolimits}
\newcommand*{\cosec}{\mathop{\operator@font cosec}\nolimits}
\newcommand*{\arctg}{\mathop{\operator@font arctg}\nolimits}
\newcommand*{\arcctg}{\mathop{\operator@font arcctg}\nolimits}
\newcommand*{\sh}{\mathop{\operator@font sh}\nolimits}
\newcommand*{\ch}{\mathop{\operator@font ch}\nolimits}
\makeatother

\newcommand{\REM}[1]{\relax}

\numberwithin{equation}{section}

\newcommand{\Hol}{{\sf Hol}}

\newcommand{\UH}{\mathbb{H}}

\newcommand{\Real}{\mathbb{R}}

\newcommand{\Complex}{\mathbb{C}}

\newcommand{\ComplexE}{\overline{\mathbb{C}}}

\newcommand{\UD}{\mathbb{D}}
\newcommand{\clS}{\mathcal{S}}

\newcommand{\Sgk}{\Sigma(k)}

\newcommand{\onto}
{\xrightarrow{\scriptstyle \!\mathsf{onto}\,}}

\newcommand{\into}
{\xrightarrow{\hbox{\lower.2ex\hbox{$\scriptstyle \smash{\mathsf{into}}$}}\,}}

\let\C=\Complex

\newcommand{\STOP}{\par\hbox to\textwidth{\color{red}\leaders\hbox{\,STOP\,}\hfil}\par}

\newcommand{\mcite}[1]{\csname b@#1\endcsname}

\newcommand{\diam}{\mathop{\mathsf{diam}}\nolimits}

\newtheorem{result}{Theorem}
\def\dist{{\sf dist}}
\def\len{{\sf length}}

\def\Re{\mathop{\mathtt{Re}}}
\def\Im{\mathop{\mathtt{Im}}}

\newcommand{\sgn}{\mathop{\mathrm{sgn}}}

\newtheorem{theorem}{Theorem}
\newtheorem{proposition}{Proposition}[section]
\newtheorem{corollary}[proposition]{Corollary}
\newtheorem{lemma}[proposition]{Lemma}

\theoremstyle{definition}
\newtheorem{definition}{Definition}

\theoremstyle{remark}
\newtheorem{remark}{Remark}[section]

\numberwithin{proposition}{section}
\numberwithin{equation}{section}

\newcommand{\di}{\mathrm{d}}

\def\mydot#1{\smash{\stackrel{\,\lower.12ex\hbox{\text{\LARGE.}}}{#1}}\vphantom{\raise.2ex\hbox{$#1$}}}

\DeclareMathOperator*{\esssup}{ess\,sup}

\newcommand{\SHOWCORRECTIONS}{%
\newcommand{\nv}[1]{{\color{green!60!black}##1}}%
\newcommand{\comm}[1]{{\color[rgb]{0.5,0,0.5}##1}}%
\newcommand{\dv}[1]{{\color[rgb]{0.65,0.74,0.79}\sout{##1}}}%
\newcommand{\IN}[1]{{\color[rgb]{1.00,0.33,0.33}##1}}
\newcommand{\ID}[1]{{\color[rgb]{0.65,0.55,0.62}\sout{##1}}}%
\newcommand{\IC}[1]{{\color[rgb]{0.00,0.00,1}##1}}%
\newcommand{\RD}[1]{\textcolor{red}{##1}}%
}%

\newcommand{\HIDECORRECTIONS}{%
\newcommand{\nv}[1]{##1}
\newcommand{\dv}[1]{\relax}%
\newcommand{\comm}[1]{\relax}%
\let\IN=\nv%
\let\ID=\dv%
\let\IC=\comm%
\newcommand{\RD}[1]{##1}%
}%

\SHOWCORRECTIONS

        \let\bibciteOLD=\bibcite
        \renewcommand{\bibcite}[2]{\bibciteOLD{#1}{\textsc{#2}}}
        \makeatletter
        \renewcommand{\@biblabel}[1]{[\textsc{#1}]\hfill}
        \makeatother

\title{On existence of Becker extension}

\author[P. Gumenyuk]{Pavel Gumenyuk}

\address{Department of Mathematics, Politecnico di Milano, via E. Bonardi 9, 20133~Milan, Italy}
\email{pavel.gumenyuk@polimi.it}

\date{\today}

\let\le=\leqslant
\let\ge=\geqslant

\newcommand{\rc}{\mathsf{R_{c}}}

\begin{document}

\begin{abstract}
A well-known theorem by J.\,Becker states that if a normalized univalent function~$f$ in the unit disk~$\mathbb{D}$ can be embedded as the initial element into a Loewner chain~$(f_t)_{t\geqslant 0}$ such that the Herglotz function $p$ in the Loewner\,--\,Kufarev PDE $$\partial f_t(z)/\partial f=zf'_t(z)p(z,t),\qquad z\in\UD,\text{~a.e.\,}t\ge0,$$ satisfies $\big|(p(z,t)-1)/(p(z,t)+1)\big|\le k<1$, then $f$ admits a $k$-q.c. (~=``$k$-quasiconformal'') extension $F:\mathbb{C}\to\mathbb{C}$. The converse is not true. However, a simple argument shows that if~$f$ has a $q$-q.c. extension with $q\in(0,1/6)$, then  Becker's condition holds with~$k:=6q$. In this paper we address the following problem: find the largest $k_*\in(0,1]$ with the property that for any $q\in(0,k_*)$ there exists $k_0(q)\in(0,1)$ such that every normalized univalent function~$f:\mathbb D\to\mathbb C$ with a $q$-q.c. extension to~$\mathbb C$ satisfies Becker's condition with $k:=k_0(q)$. We prove that $k_*\ge1/3$.
\end{abstract}

\maketitle

\section{Introduction}
Univalent functions admitting a quasiconformal extension is one of the classical topics in Geometric Function Theory closely related to Teichm\"uller Theory, see e.g. \cite{Becker80, Krushkal-handbook, Takhtajan}.
In~1972, J. Becker~\cite{Becker72, Becker76} found a witty construction of q.c.-extensions for holomorphic functions based on Loewner's parameteric method. Although Becker's extensions are quasiconformal mappings of quite special nature~\cite[Theorem~2]{Istvan}, his result is interesting from several points of view. Precise definitions and a brief discussion on this matter can be found in Sect.\,\ref{SS_Becker}.

Taking advantage of Becker's result, the author and I.\,Hotta~\cite{PAMS2020} answered recently a question concerning the sharp bound of the third coefficient raised in~1977 by K\"uhnau and Niske~\cite{KuhnauNiske}. In the same paper~\cite{PAMS2020} the following problem was stated.
\vskip1ex
\noindent\textbf{Problem.} Find the largest $k_*\in(0,1]$ such that there exists a function $k_0:(0,k_*)\to(0,1)$ with the following property: if $q\in(0,k_*)$, then any univalent function $f:\UD\to\Complex$ admitting a $q$-q.c. extension to~$\Complex$ has also a $k_0(q)$-q.c. Becker extension.
\vskip1ex

A simple observation concerning conditions of q.c.-extendibility in terms of the pre-Schwarzian shows that $k_*\ge1/6$, see \cite[Sect.\,5]{PAMS2020}. An analogous but more involving argument~\cite[Corollary~6.7]{PAMS2020} allows one to achieve a slight improvement of this result.

On the other hand, for $q\in(0,1)$ close to~$1$, it is not known whether every univalent function~$f$ admitting a $q$-q.c. extension to~$\Complex$ has also a $k$-q.c. Becker extension with some $k\in(0,1)$, even if $k$ is allowed to depend on~$f$. In fact, it was shown~\cite[Proposition~5.2]{PAMS2020} that two natural ways to construct a Loewner chain starting from such an~$f\,$ fail to produce Becker q.c.-extensions in general. The main result of this paper is as follows.
\begin{theorem} \label{TH_MAIN}
In the above notation, $k_*\ge1/3$; i.e. for any $q\in(0,1/3)$ there exists $k_0\in(0,1)$ depending only on~$q$ such that every $f\in\clS_q$ admits a $k_0$-q.c. Becker extension.
\end{theorem}
An optimistic conjecture is that~$k_*=1$. At the same time, it would not be surprising if $k_*$ turned out to be equal to~$1/3$.

Note that the quasiconformal extensions~$F$ of holomorphic functions~$\UD$ considered in the above theorem satisfy $F(\infty)=\infty$. This condition is not merely a normalization. In fact, it turns out to be quite essential, see e.g. Remark~\ref{RM_infinity}. An analogue of the above theorem for quasiconformal extensions of functions $g\in\Sigma$ follows easily from known results, but its does not imply Theorem~\ref{TH_MAIN} (which was unfortunately overlooked in \cite[Sect.\,5]{Istvan}).

The proof of Theorem~\ref{TH_MAIN} is given in Sect.\,\ref{SS_main-proof}. It is preceded by a block of auxiliary statements in Sect.\,\ref{SS-lemmas}\,--\,\ref{SS_dist-diam}. To facilitate the reader, we explain the main idea and give a short overview of the main steps of the proof in Sect.\,\ref{SS_main-construction}.

Necessary details on the Parametric Method and Becker's extensions are given the next section. The Loewner chain we construct in the proof of Theorem~\ref{TH_MAIN} is defined as the family of conformal mappings onto domains bounded by images of concentric circles under a certain diffeomorphic mapping. In Sect.~\ref{S_LCh-diffeos} we combine results from~\cite{OE} and~\cite{Schippers2005} to show that such families are differentiable w.r.t. the parameter and satisfy the Loewner\,--\,Kufarev~PDE.

\section{Preliminaries}
\subsection{Loewner chains and Becker's extensions}\label{SS_Becker}
An important role in the study of univalent functions is played by the Parametric Method, going to back Loewner~\cite{Loewner}.

Let $p:\UD\times[0,+\infty)\to\Complex$ be a (classical) \textsl{Herglotz function}, i.e.
\begin{itemize}
  \item[HF1:] for each~$z\in\UD$, $p(z,\cdot)$ is locally integrable on~$[0,+\infty)$;
  \item[HF2:] for a.e.~$t\ge0$, $p(\cdot,t)$ is holomorphic in~$\UD$ and satisfies $\Re p(\cdot,t)\ge0$;
  \item[HF3:] $\int\limits_0^{+\infty}\Re p(0,t)\,\di t=+\infty$.
\end{itemize}
Usually instead of HF3, one assumes the normalization~$\Re p(0,t)=1$ for a.e.\,${t\ge0}$.  Following Becker~\cite{Becker76}, \cite[Sect.\,5.1]{Becker80}, as a matter of convenience we prefer to work with a less restrictive condition~HF3. With a suitable slight modifications, most of the classical results remain valid, because the normalization can be achieved by a change of variables. In particular, it is known, see e.g. \cite[Theorem~6.3 on p.\,160]{Pommerenke}, that
the initial value problem for the Loewner\,--\,Kufarev ODE
\begin{equation}\label{EQ_LK-ODE0}
{\di w}/{\di t}=-w\,p(w,t),\quad t\ge0,\qquad w(z,0)=z\in\UD,
\end{equation}
has a unique solution $w=w(z,t)$ and that the limit
\begin{equation}\label{EQ_f-limit}
f(z):=\lim_{t\to+\infty}\frac{\,\mathrlap{w(z,t)}\hphantom{w'(0,t)}}{\,w'(0,t)},\quad z\in\UD,
\end{equation}
exists locally uniformly in~$\UD$ and belongs to the class~$\clS$ of all univalent functions $f:\UD\to\Complex$ normalized by~${f(0)=0}$, ${f'(0)=1}$.

Moreover, a much deeper result, see~\cite{Gut} and \cite[Theorem~6.1 on p.\,160]{Pommerenke}, states that for any $f\in\clS$ there exists a Herglotz function~$p$, not necessarily unique, such that $f$ is given by~\eqref{EQ_f-limit}. The proof of this fact is based on the possibility to embed~$f$ as the initial element into a suitable Loewner chain~$(f_t)$.

\begin{definition}
A family $(f_t)_{t\ge0}$ of holomorphic functions is called a (radial) \textsl{Loewner chain} in~$\UD$ if it satisfies the following conditions:
\begin{itemize}
\item[LC1:] for each~$t\ge0$, $f_t:\UD\to\Complex$ is univalent, with $f_t(0)=0$;
\item[LC2:] for any $s,t\ge0$ with $s<t$, $f_s(\UD)\subset f_t(\UD)$;
\item[LC3:] the function $t\mapsto f'_t(0)$ is locally absolutely continuous in~$[0,+\infty)$, $f'_0(0)=1$, and $$\lim_{t\to+\infty}|f_t'(0)|\to+\infty.$$
\end{itemize}
\end{definition}

Every Herglotz function~$p$ generates a Loewner chain~$(f_t)$ by
\begin{equation}\label{EQ_LoewnerChain}
  f_s(z):=\lim_{t\to+\infty}\frac{\,\mathrlap{w(z;s,t)}\hphantom{w'(0;0,t)}}{\,w'(0;0,t)},\quad z\in\UD,~s\ge0,
\end{equation}
where for any $s\ge0$, $t\ge s$, and $z\in\UD$, $w(t)=w(z;s,t)$ stands for the unique solution to
$$
{\di w}/{\di t}=-w\,p(w,t),\quad t\ge s,\qquad w(z;s,s)=z\in\UD.
$$
The Loewner chain $(f_t)$ defined in this way satisfies the Loewner\,--\,Kufarev PDE:
$$
\partial f_t(z)/\partial t=zf_t'(z)p(z,t),\quad z\in\UD,~t\ge0.
$$

Furthermore, the correspondence between the Herglotz functions~$p$ and Loenwer chains defined by~\eqref{EQ_LoewnerChain} is a bijection, i.e. for any Loewner chain~$(f_t)$ there exists a unique Herglotz function~$p$ such that equality~\eqref{EQ_LoewnerChain} holds, see e.g. \cite[\S6.1]{Pommerenke65}.

Thus, the class $\clS$ can be seen as the image of the set of all Herglotz functions in~$\UD$ under the map defined via the Loewner\,--\,Kufarev ODE~\eqref{EQ_LK-ODE0} and formula~\eqref{EQ_f-limit}. It is natural to ask which properties of a Herglotz function~$p$ ensure that the corresponding function~$f=f_0$ belongs to a given subclass of~$\clS$. One important result in this direction was discovered in~1972 by Becker~\cite{Becker72}. To state Becker's theorem, we need the following definition.
\begin{definition}
Let  with $k\in[0,1)$. We say that a univalent (holomorphic or meromorphic) function $f$ defined in a domain $D\subset\ComplexE$ \textsl{admits a $k$-q.c. extension to~$\ComplexE$} if there exists a $k$-q.c. mapping $F:\ComplexE\to\ComplexE$ such that $f=F|_D$. Moreover, a holomorphic function $f:D\to\Complex$, $D\subset\Complex$, is said to \textsl{admits a $k$-q.c. extension to~$\Complex$} if there exists a $k$-q.c. mapping $F:\Complex\to\Complex$ such that $f=F|_D$.
\end{definition}
For the theory of quasiconformal mappings in the plane, we refer to e.g. \cite{Ahlfors_lect} or~\cite{Lehto_book}.
\begin{definition}
Denote by \emph{$\clS_k$} the class of all $f\in\clS$ admitting $k$-q.c. extensions to~$\Complex$.
\end{definition}
\begin{remark}\label{RM_infinity}
Due to the fact that all isolated singularities of quasiconformal mappings are removable, the q.c.-extendibility of a function~$f$ to~$\Complex$ is in fact a bit stronger condition than the q.c.-extendibility to~$\ComplexE$: in addition to existence of a $k$-q.c. extension $F:\ComplexE\to\ComplexE$, it is required that $F(\infty)=\infty$.
As a result, certain properties of the class $\clS_k$  are different from those of the wider class formed all~$f\in\clS$ having $k$-q.c. extensions to~$\ComplexE$. In particular,  $\clS_k$ admits Lehto's Majorant Principle~\cite{Lehto_Majorant}, while the latter class does not. Moreover, if~$f\in\clS_k$, then $|f(z)|< M(k)\le 4^{K-1}$, $K:=(1+k)/(1-k)$, for all~$z\in\UD$, see \cite{Kuhnau1975, Gut1978, Gokturk}, but this is not necessarily the case if $f\in\clS$ admits a $k$-q.c. extension to~$\ComplexE$.
\end{remark}

\begin{result}[Becker \protect{\cite{Becker72,Becker76}}]\label{TH_Becker}
Let $k\in[0,1)$ and let $(f_t)$ be a radial Loewner chain whose Herglotz function~$p$ satisfies
\begin{equation}\label{EQ_Beckers-condition}
  p(\UD,t) \subset U(k):=\left\{ w \in \C \colon \left|\frac{w-1}{w+1}\right| \le k\right\}\quad\text{for  a.e.~$t\ge0$}.
\end{equation}
Then for every~$t\ge0$, the function~$f_t$ admits a $k$-q.c. extension to~$\Complex$. In particular, such an
extension for~$f_0$ is given by
\begin{equation}\label{EQ_BeckerExt}
F(\rho e^{i\theta}):=\left\{
 \begin{array}{ll}
  f_0(\rho e^{i\theta}), & \text{if~$0\le \rho<1$},\\[1ex]
  f_{\log \rho}(e^{i\theta}), & \text{if~$\rho\ge1$}.
 \end{array}
\right.
\end{equation}
\end{result}
\begin{definition}
  The extension~$F$ of $f=f_0$ given in the above theorem is called a \textsl{Becker extension} of $f\in\clS$.
\end{definition}
\begin{remark}
Theorem~\ref{TH_Becker} has a sort of converse~\cite[Theorem~2\,(II)]{Istvan}: if $(f_t)$ is a Loewner chain such that for any $t\ge0$, $f_t$ extends continuously to~$\partial\UD$ and if $F:\Complex\to\Complex$ defined by~\eqref{EQ_BeckerExt} is a $k$-q.c. mapping, then the Herglotz function~$p$ of the Loewner chain~$(f_t)$ satisfies~\eqref{EQ_Beckers-condition}.
\end{remark}

Becker's condition for q.c.-extendibility given in the above theorem is sufficient, but it is not necessary, see e.g. \cite[Theorem~3]{Istvan}. However, most of the known sufficient conditions can be deduced from Theorem~\ref{TH_Becker}, see e.g. \cite[Sect.\,5.3--5.4]{Becker80}, \cite{Ikkei09, Ikkei11}, and~\cite{Pommerenke86}. Moreover, Becker's condition~\eqref{EQ_Beckers-condition} remains to be sufficient for q.c.-extendibility in other variants of the Parametric Method, see e.g.~\cite{HottaGum::QC-chordal}, \cite[Theorem~1]{Istvan}, and~\cite{Ikkei2}. Finally, there are many examples, see e.g. \cite[Proposition~4.2]{PAMS2020}, in which Becker's extension is the best possible in the sense that it has the smallest possible sharp upper bound $\esssup_{|z|>1}|\mu_F(z)|$ for the Beltrami coefficient~$\mu_F$ among all q.c.-extensions~$F:\Complex\to\Complex$ of a given $f\in\clS$.

Therefore, the study of Becker's extensions represents a considerable interest. One of the questions to investigate is whether any  $f\in\clS$ with a q.c.-extension to~$\Complex$ admits also a Becker q.c.-extension. From the main result of this paper, Theorem~\ref{TH_MAIN}, it follows that the answer is positive at least for all $q$-q.c. extendible functions with $q\in(0,1/3)$. The main difficulty resides however in proving existence of $k$-q.c. Becker extensions with some $k=k_0(q)\in(0,1)$ depending only on~$q$ but not on the function~$f$.

\subsection{Loewner Chains in~$\Delta$}
In some cases, it is more convenient to work with univalent functions in $\Delta:=\ComplexE\setminus\overline\UD$ normalized by $g(\infty)=\infty$. Most of the classical Loewner Theory extends easily to this case. In particular, we say that $(g_t)_{t\ge0}$ is a (radial) \textsl{Loewner chain in~$\Delta$} if:
\begin{itemize}
\item[(i)] for any $t\ge0$, $g_t$ is a univalent meromorphic function in~$\Delta$ with $g_t(\infty)=\infty$;
\item[(ii)] for any $s\ge0$ and any $t\ge s$, $g_s(\Delta)\subset g_t(\Delta)$;
\item[(iii)] the function $t\mapsto g'_t(\infty)\in\Complex$ is locally absolutely continuous in~$[0,+\infty)$, with $g'_0(\infty)=1$ and $$\lim_{t\to+\infty}g'_t(\infty)=0,$$ where as usual by $g_t'(\infty)$ we mean the coefficient of~$z$ in the Laurent expansion of $g_t(z)$ in~${\Delta\setminus\{\infty\}}$.
\end{itemize}
Moreover, a function $p:\Delta\times[0,+\infty)$ is said to be a \textsl{Herglotz function in~$\Delta$}, if $(\zeta,t)\mapsto p(1/\zeta,t)$ is a Herglotz function in~$\UD$.

If $(g_t)$ is a Loewner chain in~$\Delta$, then $$\bigcap\nolimits_{t\ge0}\Complex\setminus g_t(\Delta)$$ is non-empty. In fact, this set contains exactly one point, which we denote by~$w_0$. The functions $f_t(\zeta):=\big(g_t(1/\zeta)-w_0\big)^{-1}$, $\zeta\in\UD$, form a Loewner chain in~$\UD$. It follows that there exists a null-set $N\subset[0,+\infty)$ such that the limit
$$
\frac{\partial g_t(z)}{\partial t}:=\lim_{s\to t}\frac{g_s(z)-g_t(z)}{s-t}
$$
exists locally uniformly in~$\Delta\setminus\{\infty\}$. Moreover, using the Loewner\,--\,Kufarev PDE for $(f_t)$ we easily see that $(g_t)$ satisfies
\begin{equation}\label{EQ_LK-PDE-in-Delta}
\frac{\partial g_t(z)}{\partial t}=-zg_t'(z)p(z,t),\quad\text{a.e.\,$t\ge0$,~}~z\in\Delta\setminus\{\infty\},
\end{equation}
for a suitable Herglotz function~$p$ in~$\Delta$, which is uniquely defined for all $t\in[0,+\infty)\setminus N$.

Conversely, if $p$ is a Herglotz function in~$\Delta$, then there exists a Loewner chain~$(g_t)$ in~$\Delta$, satisfying~\eqref{EQ_LK-PDE-in-Delta}. The only difference from the case of~$\UD$, which one has to keep in mind, is that to each Herglotz function there corresponds a one-parameter family of Loewner chains~$(g_t)$ that differ from each other by an additive constant.

\section{Subordination chains defined by diffeomorphisms}\label{S_LCh-diffeos}
In this section we establish an auxiliary assertion concerning the Herglotz function of a Loewner chain, assuming that the boundaries $\partial f_t(\UD)$ and their evolution in~$t$ is sufficiently regular. Since this result might have some independent interest, we state and prove it separately from the main discussion of this paper.

Let $\Psi:\{e^{t+i\tau}:t\in(a,b),\,\tau\in[0,2\pi]\}\to\Complex$ be a homeomorphic map such that for any $s,t\in(a,b)$, $s<t$, the curve $\Gamma_s:=\Psi(\{z:|z|=e^s\})$ is contained in the Jordan domain~$\Omega_t$ bounded by~$\Gamma_t$. Replacing $\Psi$ with $z\mapsto\Psi(\bar z)$ if necessary, we may suppose that for each~$t\in(a,b)$ the parametrization $[0,2\pi]\ni\tau\mapsto\Psi(e^{t+i\tau})$ induces on~$\Gamma_t$ the positive (i.e. counterclockwise) orientation. Finally, using translations we may assure that $0\in\Omega_t$ for all~$t\in(a,b)$.

For each $t\in(a,b)$, let $f_t$ be the conformal mapping of~$\UD$ onto $\Omega_t$ and normalized by $f_t(0)=0$, $f_t'(0)>0$. Then the family $(f_t)_{t\in(a,b)}$ is a subordination chain. We combine results from~\cite{Schippers2005} and~\cite{OE} to show that if $\Psi$ is regular enough, then $(f_t)$ is differentiable for all~$t\in(a,b)$ and satisfies the Loewner\,--\,Kufarev equation with a Herglotz function having a continuous extension to~$\partial\UD$.

\begin{proposition}\label{PR_C2-homotopy}
If $\Psi$ is $C^2$-diffeomorphic (i.e. it is of class $C^2$ and its Jacobian determinant does not vanish), then for any $t\in(a,b)$ the limit
\begin{equation}\label{EQ_f_t-time-derivative}
\frac{\partial f_t(z)}{\partial t}=\lim_{s\to t}\frac{f_s(z)-f_t(z)}{s-t}
\end{equation}
exists locally uniformly in~$\UD$. Moreover, $(f_t)$ satisfies the Loewner\,--\,Kufarev PDE
\begin{equation}\label{EQ_PDE-prop}
\frac{\partial f_t(z)}{\partial t}=zf'_t(z)p(z,t),
\end{equation}
where $p(\cdot,t)$ is a holomorphic function in~$\UD$ with positive real part and continuous extension to $\partial\UD$ which is  uniquely determined by $\Im p(0,t)=0$ and
\begin{equation}\label{EQ_Re-p-on-boundary}
  \Re p(e^{i\theta},t)=-\frac{1}{|f_t'(e^{i\theta})|}\Im\left(\overline{\frac{\partial\Psi(e^{t+i\tau})}{\partial \tau}}\,\frac{\partial\Psi(e^{t+i\tau})}{\partial t}\right)\left|\frac{\partial\Psi(e^{t+i\tau})}{\partial \tau}\right|^{-1}
\end{equation}
for all $\theta\in[0,2\pi]$ with $\tau=\tau(\theta)$ satisfying $f_t(e^{i\theta})=\Psi(e^{t+i\tau(\theta)})$.
\end{proposition}
\begin{remark}
  An argument used  in the proof of Theorem~\ref{TH_MAIN} shows that in conditions of the above proposition a stronger assertion holds. Namely, the functions $(z,t)\mapsto \log |f_t'(z)|$ and $p$ are bounded on compact subsets of $\overline\UD\times(a,b)$. It follows that $t\mapsto f_t\in\Hol(\UD,\Complex)$ is locally absolutely continuous on~$(a,b)$ and hence $(f_t)$ is a \textit{solution} to the Loewner\,--\,Kufarev equation~\eqref{EQ_PDE-prop}; for a precise definition, see e.g.\,\cite[Definition~2.1]{duality}.
\end{remark}
\begin{proof}[\textbf{Proof of Proposition~\protect{\ref{PR_C2-homotopy}}}]
Since $\Gamma_t$'s are of class $C^2$, the functions $f_t$ extend $C^1$-smoothly to $\partial\UD$ and $f'_t$ does not vanish in~$\overline\UD$, see e.g. \cite[Theorem~3.5 on p.\,48]{Pommerenke:BB}.

Fix some $t\in(a,b)$.
Following~\cite{Schippers2005} denote by $n_u$ and $\mathcal L_u$,  $u\in\Gamma_{t}$, the outward normal unit vector and the normal line to the curve $\Gamma_{t}$ at the point~$u$, respectively. Furthermore, for $s\in(a,b)$, $s\neq t$, we denote by $\Omega_{s,t}$ the doubly connected domain bounded by $\Gamma_s$ and $\Gamma_t$  and let $I(s,\tau)$ be the unique connected component of $\mathcal L_u\cap\Omega_{s,t}$ for which $u$ is one of the end-points. Denote by $w=w(s,u)$ the other end-point. If $|t-s|$ is small enough, then $w(s,u)$ lies on~$\Gamma_s$ and moreover, the map $u\mapsto w(s,u)$ is a bijection between $\Gamma_t$ and $\Gamma_s$.  Write $\Delta n_u(s):=\sgn(s-t)\big|u-w(s,u)\big|$. It is not difficult to see that $w(s,u)$ is differentiable in $s$. Therefore,
$$
\nu(u):=\left.\frac{\di\Delta n_u(s)}{\di s}\right|_{s=t}
$$
is well-defined for all~$u\in\Gamma_t$.

Denote by $g_s$ the Green function of~$\Omega_s$. By \cite[Theorem~1]{Schippers2005}, the function $t\mapsto g_s(z,0)$ is differentiable at~$s=t$ locally uniformly w.r.t. $z\in\Omega_{t}$ and
\begin{equation}\label{EQ_Schippers}
  h_{t}(z):=\left.\frac{\di}{\di s}\right|_{s=t}\!\! g_s(z,0)=\frac1{2\pi}\int_{\Gamma_{t}}\frac{\partial g_{t}(u,z)}{\partial n}\frac{\partial g_{t}(u,0)}{\partial n}\nu(u)\,\big|\di u\big|,
\end{equation}
where $\partial/\partial n$ stands for the derivative along the outward normal direction to~$\Gamma_{t}$ and $|\di u|$~is the length element of~$\Gamma_{t}$.

The r.h.s. of~\eqref{EQ_Schippers} is a Poisson integral in the domain~$\Omega_{t}$. Taking into account regularity of the boundary, we conclude that $h_{t}(z)$ is harmonic in~$\Omega_{t}$ and continuous on its closure, with
\begin{equation}\label{EQ_h-on-boundary}
  h_{t}(u)=\frac{\partial g(u,0)}{\partial n}\nu(u)\quad\text{for all~$u\in\Gamma_{t}$.}
\end{equation}

Notice that
 $\nu(u)$ coincides with the projection of $\partial\Psi/\partial t$ onto $n_u$. More precisely,
\begin{equation}\label{EQ_normal-velocity}
  \nu(\Psi(e^{t+i\tau}))=-\Im\left(\overline{\frac{\partial\Psi(e^{t+i\tau})}{\partial \tau}}\, \frac{\partial\Psi(e^{t+i\tau})}{\partial t}\right)\left|\frac{\partial\Psi(e^{t+i\tau})}{\partial \tau}\right|^{-1}
\end{equation}
for all $\tau\in[0,2\pi]$. Indeed, fix some $\tau\in[0,2\pi]$ and let $\tau_s$ be a  solution to $\Psi(e^{s+i\tau_s})=w(s,u)$ with $u:=\Psi(e^{t+i\tau})$, which is clearly unique modulo~$2\pi$. Then $\tau_t=\tau$ and $s\mapsto\tau_s$ is differentiable. Denote $v:=\big(\di\tau_s/\di s\big)|_{s=t}$. By the Chain Rule,
\begin{multline*}
\nu(\Psi(e^{t+i\tau}))=\overline{n_u}\,\left.\frac{\partial w(s,u)}{\partial s}\right|_{s=t}=\Re\left(\overline{n_u}\,\left.\frac{\partial w(s,u)}{\partial s}\right|_{s=t}\right)=\\
=\Re\left(\overline{n_u}\,\left(\frac{\partial\Psi(e^{t+i\tau})}{\partial t}+v\frac{\partial\Psi(e^{t+i\tau})}{\partial \tau}\right)\right)=\Re\left(\overline{n_u}\,\frac{\partial\Psi(e^{t+i\tau})}{\partial t}\right),
\end{multline*}
where the last equality holds because $\partial\Psi/\partial\tau$ is orthogonal to~$n_u$. To obtain~\eqref{EQ_normal-velocity} it remains to substitute $n_u=-i\partial\Psi/\partial\tau\big|\partial\Psi/\partial\tau\big|^{-1}$.

It is shown in~\cite[Proof of Theorem~6]{OE} that differentiability of the Green function $g_t$ w.r.t. the parameter~$t$ implies existence of the locally uniform limit~\eqref{EQ_f_t-time-derivative} and that in such a case $(f_t)$ satisfies the Loewner\,--\,Kufarev equation with $p(\cdot,t)=H_{t}\circ f_t$, where $H_t$ is the holomorphic function in~$\Omega_t$ satisfying $\Im H_t(0)=0$ and $\Re H_t=h_t$.

Note that ${\partial g(u,0)}/{\partial n}=|(f^{-1}_t)'(u)|$, which is a function of class $C^1$ defined on~$\Gamma_t$. Using formulas~\eqref{EQ_h-on-boundary} and~\eqref{EQ_normal-velocity}, we see that $h_t$ extends  $C^1$-smoothly to~$\partial\UD$. By Privalov's Theorem, it follows that $H_t$ and hence $p(\cdot,t)$ extend continuously to the boundary. Combining now \eqref{EQ_h-on-boundary} and~\eqref{EQ_normal-velocity}, we obtain~\eqref{EQ_Re-p-on-boundary}.
\end{proof}

\section{Proof of the Main Result}
As usual we will denote by $\Sigma$ the class of all univalent functions $g$ in $\Delta:=\ComplexE\setminus\overline\UD$ with the expansion of the form $g(z)=z+b_0+b_1/z+b_2/z^2+\ldots$ and let $\Sgk$, $k\in(0,1)$, stands for the class of all $g\in\Sigma$ admitting $k$-q.c. extentions to~$\ComplexE$.

\subsection{Notation and the main construction}\label{SS_main-construction}
In this section we explain the ideas used in the proof of Theorem~\ref{TH_MAIN}.
Fix $q\in(0,1/3)$ and $f\in\clS_q$, $f(\zeta)=\zeta+a_2(f)\zeta^2+\ldots$ Then $g_0(z):=1/f(1/z)$ belongs to $\Sigma(q)$. Moreover,
\begin{equation}\label{EQ_Schw-estimate}
  (1-|\zeta|^2)^2|S_f(\zeta)|\le 6q\qquad \text{for all $\zeta\in\UD$,}
\end{equation}
see \cite[Satz~$3^*$]{Kuhnau69}  or \cite[Example~9 on p.\,134]{Lawrynowicz}. It follows that $g_0$ has a $k$-q.c. extension $G:\ComplexE\setminus\{0\}\to\ComplexE$, with $k:=3q$, given explicitly by
\begin{gather}\label{EQ_g-extension}
  G(z):=\begin{cases}
          g_0(z), & \mbox{if } \zeta\in\Delta, \\
          g_t(z/|z|),~t:=-\log|z|, & \mbox{if } 0<|z|\le1,
        \end{cases}\\
  \intertext{where~} \label{EQ_g_t-explicit} 1/g_t(w):=f(e^{-t}/w)+\frac{(1-e^{-2t})f'(e^{-t}/w)}{e^{-t}w-\frac12(1-e^{-2t})P_f(e^{-t}/w)},~w\in\Delta,~t\ge0,
\end{gather}
and $P_f(\zeta):=f''(\zeta)/f'(\zeta)$, $\zeta\in\UD$, is the so-called pre-Schwarzian of~$f$.

The above well-known extension is originally due to Ahlfors and Weill~\cite{AW62}, see also \cite{Ahlfors74}. The relation to Loewner chains was discovered by Becker, see \cite[Sect.\,4]{Becker72} and \cite[Sect.\,5]{Becker80}.

\begin{remark}\label{RM_real-analytic}
Note that every isolated singularity of a q.c.-map is removable; hence in fact, $G$ is a $k$-q.c. automorphism of~$\ComplexE$. Moreover, expressing $f$ via~$g_0$ and substituting $e^t w=1/\bar z$, $e^{-2t}=z\bar z$ in the above formulas, we can rewrite~\eqref{EQ_g-extension} for $z\in\Delta$ as follows:
\begin{equation}\label{EQ_explicit-zeta-zetabar}
  G(z)=g_0(1/\bar{z})- \frac{(1-z\bar{z})g_0'(1/\bar{z})}{\bar{z}+\tfrac12(1-z\bar{z})P_{g_0}(1/\bar{z})} = b_0+z+3b_1\bar{z}+6b_2\bar{z}^2+\ldots,
\end{equation}
where $b_j$'s are the Laurent coefficients of~$g_0$, i.e. $g_0(w)=w+b_0+b_1/w+b_2/w^2+\ldots$
In~particular, it follows that $G$ is real-analytic in~$\UD$, including the point~$0$.
\end{remark}

The extension~\eqref{EQ_g-extension} can be obtained with the help of Becker's construction. However, $G(0)=-a_2(f)$. Therefore, if $a_2(f)\neq0$, then it does provide a Becker extension for $\hat f(z):=f(z)/\big(1+a_2(f)f(z)\big)$, ${z\in\UD}$, with the Loewner chain $\hat f_t(z):=1/\big(g_t(z)+a_2(f)\big)$, but not for the function~$f$ itself.

In this paper we show that the family~$(g_t)$ can be modified in such a way that it defines a $k_0$-q.c. extension of~$g$, with some ${k_0\in(0,1)}$, having a fixed point at the origin. This would yield the desired Becker extension of~$f$.

The idea is as follows. Denote by $D_t$ the Jordan domain bounded by $\Gamma_t:=g_t(\partial\UD)$. It is known that $|f(z)|<4^{Q-1}$, $Q:=(1+q)/(1-q)$, for all $z\in\UD$ and any $f\in\clS_q$, see e.g.~\cite{Gokturk}\footnote{This estimate is not sharp. See~\cite{Kuhnau1975, Gut1978} for sharp estimates.}.  It follows that $0\in D_0$. Hence, there exists also $t_1>0$ such that $0\in D_{t_1}$. Furthermore, we fix some $t_0\in(0,t_1)$. Note also that $-a_2(f)\in D_t$ for all $t\ge0$. Let $L$ be a diffeomorphism of~$D_{t_1}$ onto itself that sends $-a_2(f)$ to~$0$. To have control on its properties, we choose $L$ in such a way that $L\circ G=G\circ e^{-t_0}T$ for some Moebius transformation $T$ which has no pole in~$\overline\UD$ and maps $e^{t_0-t_1}\UD$ onto itself. In particular, $L$ has a diffeomorphic extension to $D_{t_0}$ (although, in general, $L(D_{t_0})\neq D_{t_0}$). Denote by $\tilde g_t$, $t\ge t_1$, the conformal mapping of $\Delta$ onto the unbounded component of~$\Complex\setminus L(\Gamma_t)$ normalized by~$\tilde g_t(\infty)=\infty$, $\tilde g'_t(\infty)>0$. For $t\in[0,t_1]$ we set $\tilde g_t:=g_t$. Then using Proposition~\ref{PR_C2-homotopy}, it is not difficult to show that the functions $f_t(z):=1/\tilde g_t(1/z)$ form a Loewner chain in~$\UD$ starting from~$f_0=f$.

The reader might ask why we do need to fix $t_0<t_1$. As it will be clear from the proof, it is crucial to have certain control over the behaviour of~$L$ in a domain slightly larger than $D_{t_1}$. However, for a moment we may simply assume that $t_0=t_1$.

Denote by $p$ the Herglotz function of~$(f_t)$. It is easy to check that for ${t\in[0,t_1)}$,
$$
 \frac{1-p(z,t)}{1+p(z,t)}=\frac12z^2(1-e^{-2t})^2S_f(e^{-t}z),\quad z\in\UD.
$$
Taking into account~\eqref{EQ_Schw-estimate}, we see that $p$ satisfies Becker's condition~\eqref{EQ_Beckers-condition} with $k:=3q$ for all $t\in[0,t_1)$.

Moreover, there exists another suitable value of $k\in(0,1)$ such that condition~\eqref{EQ_Beckers-condition}  holds also for all $t>t_1$. Indeed, the curves $L(\Gamma_t)$ are images of the concentric circles in~$\UD$ w.r.t. the map $\Psi:=L\circ G=G\circ e^{-t_0}T$, which is a real-analytic diffeomorphism of $e^{t_0-t_1}\overline\UD$ onto~$\overline{D_{t_1}}$. Therefore, $\Re p(z,t)$ is positive and real analytic for all $(z,t)\in{\overline\UD\times(t_1,+\infty)}$, and converges to a holomorphic function with positive real part as ${t\to+\infty}$ or~${t\to t_1}$.  As a result, the values of $p$ on $\UD\times(t_1,+\infty)$ are contained in some compact set ${X\subset\UH:=}{\{w:\Re w>0\}}$. This means that \eqref{EQ_Beckers-condition} holds if $k\in(0,1)$ is chosen sufficiently close~to~$1$.

The main difficulty is to show that one can choose~$k$ depending only on~$q\in(0,1/3)$, but not on~$f\in\clS_q$. Although this is very plausible to be indeed the case, a rigourous proof requires considerable work. In particular, we need to estimate certain quantities describing the Riemann map on~$\partial\UD$ via quantities measuring the regularity of the boundary. There are many studies on the boundary behaviour of conformal mappings addressing such problems. However, most of known results contain constants depending on the Riemann map itself, while in our situation the constant may depend only on~$q$. This makes impossible to apply standard results directly.

One of the main ingredients of our proof is the following assertion.
\begin{proposition}\label{PR_subharmonic}
There exists a constant $a=a(q)>0$ depending only on~$q$ such that the function $\varphi(w):=|\Psi^{-1}(w)|^{-a(q)}$ is subharmonic in~$D_{t_0}$.
\end{proposition}
For each $t\ge t_1$, $w\mapsto\varphi(w)-e^{a(q)(t-t_0)}$ is a defining function for the domain~$\ComplexE\setminus\overline{L(D_t)}$. The latter means that it is defined in a neigbourhood of its boundary, vanishes on the boundary itself, takes negative values in the domain and positive ones in its exterior. The fact that it is subharmonic helps us to derive a lower estimate for the derivative of the conformal mapping $\tilde g_t:\Delta\to \ComplexE\setminus\overline{L(D_t)}$, $t>t_1$, on the boundary. This is Step\,1 in the proof of Theorem~\ref{TH_MAIN}, which we give in Sect.\,\ref{SS_main-proof}.

Furthermore, a somewhat similar argument, borrowed from~\cite{Range1976}, is used in Step\,2 to give an upper estimate for~$|\tilde g'_t|$.

In Step\,3, we apply Proposition~\ref{PR_C2-homotopy} to show that~$(\tilde g_t)_{t>t_1}$  satisfies the Loewner\,--\,Kufarev PDE in~$\Delta$, with the Herglotz function~$p$ continuous on~$\overline\Delta$ for each fixed~$t>t_1$. Formula~\eqref{EQ_Re-p-on-boundary} allows us to find upper and lower bounds for~$\Re p$.

In Step~4, we estimate the modulus of continuity of $\Re p$ on~$\partial\Delta$. Using the Hilbert transform on~$\partial\Delta$ we conclude that the values of $p$ lie in some compact set $X(q)\subset\UH:=\{w:\Re w>0\}$ depending only on~$q$, which is equivalent to the conclusion of the theorem.

\subsection{Estimates for the partial derivatives}\label{SS-lemmas} Keeping the notation introduced above, we establish a few estimates, which will be used in the proofs of Proposition~\ref{PR_subharmonic} and Theorem~\ref{TH_MAIN}.
\begin{lemma}\label{LM_t_star}
 For any $t\in[0,t_*)$, $t_*:=-\log(3q)$, we have $0\in D_t$.
\end{lemma}
\begin{proof}
Let $\tau$ be the smallest $t>0$ for which $0\not\in D_t$. Then $g_\tau(w)=0$ for some $w\in\partial\Delta$.  Using formula~\eqref{EQ_g_t-explicit}, we see that
\begin{equation}\label{EQ_et}
  e^{-\tau}w=\frac{1}{2}(1-e^{-2\tau})P_f(e^{-\tau}/w).
\end{equation}
The estimate $|P_\varphi(z)|\le 6/(1-|z|^2)$, $z\in\UD$, holds for all ${\varphi\in\clS}$, see e.g. \cite[Theorem~2.4 on p.\,32]{Duren}. Since $f\in\clS_q$, thanks to Lehto's Majorant Principle~\cite{Lehto_Majorant} we have ${|P_f(z)|\le 6q/(1-|z|^2)}$ for all ${z\in\UD}$. In combination with~\eqref{EQ_et} this yields the desired conclusion that ${\tau\ge t_*}$.
\end{proof}

Let us now choose $t_0:=t_*/2=-\frac12\log(3q)$.
Denote by $\partial$ and $\bar \partial$ the formal partial derivatives w.r.t.~$z$ and~$\bar z$, respectively:
$$
\partial:=\frac12\Big(\frac{\partial}{\partial x}-i\frac{\partial}{\partial x}\Big),\quad \bar\partial:=\frac12\Big(\frac{\partial}{\partial x}+i\frac{\partial}{\partial x}\Big).
$$

\begin{lemma}\label{LM_dF}
Let $F(z):=G(e^{-t_0}z)$. For any $z\in\overline\UD$, we have
\begin{equation}\label{EQ_dF}
 e^{-t_0}\frac{(1-k)^q}{(1+k^2)^2}~\le~\big|\partial F(z)\big|~\le~\frac{e^{-t_0}}{(1-k^2)^2\,(1-k)^q},\quad k:=3q.
\end{equation}
Moreover, the directional derivatives $\mathsf D_{\alpha\!}\, F(z):=\frac{\di}{\di s} F(z+se^{i\alpha})|_{s=0}$ and the Jacobian determinant~$J_f$ of~$F$ satisfy for all $z\in\overline{\UD}$ and all $\alpha\in\Real$ the following inequalities:
\begin{eqnarray}
\nonumber
 e^{-t_0}\frac{(1-k)^{1+q}}{(1+k^2)^2}&\le& \mathsf D_{*\,}F(z):=\min_{\alpha\in\Real}\big|\mathsf D_{\alpha\!}\, F(z)\big|\hphantom{~\le~} \\
\label{EQ_DF} &\le&\mathsf D^{*}F(z):=\max_{\alpha\in\Real}\big|\mathsf D_{\alpha\!}\, F(z)\big|~\le~ \frac{e^{-t_0}}{(1-k^2)(1-k)^{1+q}},
\end{eqnarray}
\begin{equation}\label{EQ_Jacobian}
  \frac{k(1-k^2)(1-k)^{2q}}{(1+k^2)^4}~\le~\big|J_F(z)\big|~\le~ \frac{k}{(1-k^2)^4\,(1-k)^{2q}}.
\end{equation}
\end{lemma}
\begin{proof}
Since $F$ is real-analytic in~$\overline\UD$, see Remark~\ref{RM_infinity}, it is sufficient to establish the estimates for $z\in\UD\setminus\{0\}$. Recall that $g_0(z)=1/f(1/z)$ belongs to~$\Sigma(q)$. Therefore,
\begin{align}
\label{EQ_Kuhnau69}
  &\big|\log g_0'(w)\big|\le q\log\frac{|w|^2}{|w|^2-1}\qquad\text{and}\\[.5ex]
\label{EQ_preScw-in-Sigma_q}
  &\big|wP_{\!g_0\,}\!(w)\big|\le 6q/(|w|^2-1) \quad\text{for all $w\in\Delta$.}
\end{align}
 Inequality~\eqref{EQ_Kuhnau69} is due to K\"uhnau~\cite[Satz~4]{Kuhnau69}. Inequality~\eqref{EQ_preScw-in-Sigma_q} can be obtained with the help of Lehto's Majorant Principle~\cite{Lehto_Majorant} from the simple estimate $\big|wP_{g}(w)\big|\le {6/(|w|^2-1)}$ valid in the whole class~$\Sigma$, which in turn follows from a more precise result due to Goluzion, see \cite{Goluzin_est} or \cite[Theorem~4 in~\S IV.3]{Goluzin}.

From~\eqref{EQ_explicit-zeta-zetabar} we obtain
\begin{equation}\label{EQ_dG}
\partial G(z)=\frac{g_0'(1/\bar{z})}{\big(1-\tfrac12(z-1/\bar z)P_{g_0}(1/\bar z)\big)^2}\quad\text{for all~$z\in\UD\setminus\{0\}$}.
\end{equation}
Recall that $\partial F(z)=e^{-t_0}\partial G(e^{-t_0}z)$.
So we replace~$z$ in~\eqref{EQ_dG} by~$e^{-t_0}z$ and apply~\eqref{EQ_Kuhnau69} with ${w:=e^{t_0}/\bar{z}}$. Taking into account that $|w|\ge e^{t_0}=\sqrt{k}$, we see that the absolute value of the numerator in~\eqref{EQ_dG} is contained between ${(1-k)^q}$ and ${(1-k)^{-q}}$.
Similarly, inequality~\eqref{EQ_preScw-in-Sigma_q} implies that the absolute value of the denominator in~\eqref{EQ_dG} is between $(1-k^2)^2$ and $(1+k^2)^2$.
This proves~\eqref{EQ_dF}.

Since $F$ is a smooth $k$-q.c. mapping of~$\UD$, $|\bar\partial F(z)|\le k|\partial F(z)|$ for all $z\in\UD$.  Therefore, to obtain \eqref{EQ_DF} and~\eqref{EQ_Jacobian} it remains to notice that
$$
|\partial F|-|\bar\partial F| \le \big|\mathsf D_\alpha\, F\big| \le |\partial F|+|\bar\partial F|\quad\text{and}\quad J_F=|\partial F|^2-|\bar\partial F|^2.
$$
The proof is complete.
\end{proof}

\begin{lemma}\label{LM_est-second-formal}
There exists a constant $M=M(q)$ depending only on~$q\in(0,1/3)$ such that
\begin{equation}\label{EQ_est-second-formal}
  \max\big\{|\partial^2 F(z)|,\,|\bar\partial^2 F(z)|,\,|\partial\bar\partial F(z)|\big\}\le M(q)|\partial F(z)|\quad\text{for all~$z\in\overline{\UD}$},
\end{equation}
where the map~$F$ is defined in Lemma~\ref{LM_dF}. In particular,
\begin{equation}\label{EQ_est-derivative-of-Jacobian}
  |\partial J_F|=|\bar\partial J_F|\le 4M(q)|\partial F|^2\quad \quad\text{for all~$z\in\overline{\UD}$}.
\end{equation}
\end{lemma}
\begin{proof}
As in the proof of Lemma~\ref{LM_dF}, it is sufficient to establish the estimates for ${z\in\UD\setminus\{0\}}$. Denote $w:=e^{t_0}/\bar z$.
Note that
by~\eqref{EQ_dF}, $\partial F$ does not vanish in~$\overline\UD$.
Hence, using equality~\eqref{EQ_dG} we obtain
$$
  \frac{\,|\partial^2 F(z)|\,}{|\partial F(z)|}=\big|\partial\log \partial F(z)\big|= e^{-t_0}\big|\partial\log \partial G(1/\bar w)\big|
=\frac{e^{-t_0}|P_{g_0}( w)|}{\Big|1+\tfrac12(w-1/\bar w)P_{g_0}(w)\Big|}.
$$
Taking into account that $|w|\ge e^{t_0}=1/\sqrt{k}$, one can use~\eqref{EQ_preScw-in-Sigma_q} to see that the denominator in the last expression is greater or equal to~$1-k^2$, while the numerator $e^{-t_0}|P_{g_0}(w)|\le2e^{-t_0}k|w|^{-1}(|w|^2-1)^{-1}\le 2k^3/(1-k)$.  This leads to
\begin{equation}\label{EQ_e1}
 \frac{\,|\partial^2 F(z)|\,}{|\partial F(z)|}\le \frac{2k^3}{(1+k)(1-k)^2}.
\end{equation}

Similarly,
\begin{equation}\label{EQ_dbar-log-d}
\frac{\,|\bar\partial\partial F(z)|\,}{|\partial F(z)|}= e^{-t_0}\big|\bar\partial\log \partial G(1/\bar w)\big|
=e^{-t_0}|w|\frac{(|w|^2-1)\big|S_{g_0}(w)\big|}{\Big|1+\tfrac12(w-1/\bar w)P_{g_0}(w)\Big|}.
\end{equation}
Taking into account that $S_{\!g_0\,}\!(w)=w^{-4}S_f(1/w)$, by~\eqref{EQ_Schw-estimate} we have
\begin{equation}
\label{EQ_Scw-in-Sigma_q}
  \big|S_{\!g_0\,}\!(w)|\le 6q/(|w|^2-1)^2=2k/(|w|^2-1)^2\quad\text{for all $w\in\Delta$.}
\end{equation}
As above, the denominator in~\eqref{EQ_dbar-log-d} is separated from zero by~$1-k^2$, while the numerator can be estimated with the help of~\eqref{EQ_Scw-in-Sigma_q}. In this way we obtain
\begin{equation}\label{EQ_e2}
 \frac{\,|\partial\bar\partial F(z)|\,}{|\partial F(z)|}=\frac{\,|\bar\partial\partial F(z)|\,}{|\partial F(z)|}\le \frac{2k^2}{(1+k)(1-k)^2}.
\end{equation}

\medskip

The estimate of $|\bar\partial^2 F(z)|/|\partial F(z)|$ is a bit more tricky.
Using~\eqref{EQ_explicit-zeta-zetabar}, we find that
\begin{equation}\label{EQ_dbarG}
\bar\partial G(1/\bar w)=-\frac{\,\tfrac12(w-1/\bar{w})^2w^2S_{g_0}(w)g_0'(w)}{\big(1+\tfrac12(w-1/\bar w)P_{g_0}(w)\big)^2}.
\end{equation}
If $S_{g_0}\equiv0$ in~$\Delta$, then $G$ is a Moebius transformation and hence $\bar\partial^2 F(z)\equiv0$ in~$\UD$. Therefore, we may suppose that $S_{g_0}$ does not vanish identically.
From \eqref{EQ_dbarG} it follows by a simple calculation that for all $w\in\Delta$ with $S_{g_0}(w)\neq0$,
\begin{multline}\label{EQ_dbar-log-dbar}
Q(w):=\bar\partial\log \bar\partial G(1/\bar w)+w^2\frac{\di}{\di w}\log\big(w^4S_{g_0}(w)\big)=\\=-\frac{2w}{|w|^2-1}+w^2\frac{(w-1/\bar w)S_{g_0}(w)}{1+\tfrac12(w-1/\bar w)P_{g_0}(w)}.
\end{multline}
Using the same technique as above, it is not difficult to see that
\begin{equation}\label{EQ_estforQ}
  |e^{-t_0}Q(w)|\le\frac{2k}{1-k}+\frac{2k^2}{(1+k)(1-k)^2}\le\frac{2k}{(1-k)^2}.
\end{equation}
Furthermore,
\begin{equation}\label{EQ_Schw-trick}
w^2\frac{\di}{\di w}\log\big(w^4S_{g_0}(w)\big)=w^2\frac{\di}{\di w}\log S_f(1/w) =-\frac{S_f'(1/w)}{S_f(1/w)}=-\frac{S_f'(1/w)}{w^4S_{g_0}(w)}.
\end{equation}
Recall that $|w|\ge e^{t_0}=1/\sqrt{k}$ and apply the Cauchy estimate for the derivative of $S_f$ in the disk of radius $(1-\sqrt{k})/2$ centered at the point~$1/w$. By~\eqref{EQ_Schw-estimate}, for all $\zeta\in\partial D$ we have $|S_f(\zeta)|\le 2k/(1-|\zeta|^2)^2\le 2k/\big(1-\tfrac14(1+\sqrt{k})^2\big)^2$.
Hence,
\begin{equation}\label{EQ_est-der-of-Schw}
  |S'_f(1/w)|\le\frac{4k}{(1-\sqrt{k})\big(1-\tfrac14(1+\sqrt{k})^2\big)^2}\le\frac{16k}{(1-k)^3}.
\end{equation}
From~\eqref{EQ_dG} and~\eqref{EQ_dbarG} we get immediately that
$$
\frac{\,\bar\partial F(z)\,}{\partial F(z)}=-\tfrac12(w-1/\bar{w})^2w^2S_{g_0}(w).
$$
Combining this equality with~\eqref{EQ_Scw-in-Sigma_q}, \eqref{EQ_Schw-trick} and~\eqref{EQ_est-der-of-Schw}, we see that
$$
  \frac{\,|\bar\partial F(z)|\,}{|\partial F(z)|}\le k ~\text{~and~}\left|w^2\frac{\di}{\di w}\log\big(w^4S_{g_0}(w)\big)\right|\cdot\frac{\,|\bar\partial F(z)|\,}{|\partial F(z)|}\,\le\,\frac{8k}{1-k}.
$$
Taking into account~\eqref{EQ_estforQ}, it follows that
\begin{eqnarray}\nonumber
\frac{\,|\bar\partial^2 F(z)|\,}{|\partial F(z)|}
&=&
e^{-t_0}\big|\bar\partial\log \bar\partial G(1/\bar w)\big|\cdot\frac{\,|\bar\partial F(z)|\,}{|\partial F(z)|}\\[.75ex]
\nonumber
&\le&%
|e^{-t_0}Q(w)|\cdot\frac{\,|\bar\partial F(z)|\,}{|\partial F(z)|}~+~e^{-t_0}\left|w^2\frac{\di}{\di w}\log\big(w^4S_{g_0}(w)\big)\right|\cdot\frac{\,|\bar\partial F(z)|\,}{|\partial F(z)|}\\[.75ex]
\label{EQ_e3}
&\le&%
\frac{2k^2}{(1-k)^2}+\frac{8k^{3/2}}{1-k}=:M(q).
\end{eqnarray}
The last inequality is obtained for all $z\in\UD$, $z\neq0$, such that $S_{g_0}(e^{t_0}/\bar z)\neq0$. Since $S_{g_0}$ is holomorphic, its zeros are isolated and hence \eqref{EQ_e3}  holds everywhere in~$\overline{\UD}$.

\medskip

Inequalities~\eqref{EQ_e1}, \eqref{EQ_e2}, and~\eqref{EQ_e3} imply~\eqref{EQ_est-second-formal}, which in turn implies~\eqref{EQ_est-derivative-of-Jacobian} since
$$
 |\bar\partial F|\le k|\partial F|
~\text{~and~}~
 \bar\partial J_f=\overline{\partial J_f}=\partial\bar\partial F\,\overline{\partial F}\,+\,
 \overline{\partial^2 F}\,\partial F\,-
 \,\bar\partial^2 F\,\overline{\bar\partial F}\,-\,\overline{\partial\bar\partial F}\,\bar\partial F,
$$
where the  equality $\bar\partial J_f=\overline{\partial J_f}$ holds because $J_F$ is real-valued.
\end{proof}

\subsection{Proof of Proposition~\ref{PR_subharmonic}}\label{SS_proof-of-proposition-subharmonic} Referring to the construction explained in Sect.\,\ref{SS_main-construction}, we start by making an appropriate choice of $t_1>t_0$ and $T$. Let us recall that by Lemma~\ref{LM_t_star}, $0\in D_t$ for all $t\in(0,t_*)$, where $t_*=-\log(3q)$. Recall also that we fixed $t_0:=t_*/2$. Therefore, for $z_0:=e^{t_0}G^{-1}(0)$ we have
\begin{equation}\label{EQ_z_0-est}
  |z_0|\le e^{-t_0}\le\sqrt{3q}.
\end{equation}

Since $L(-a_2(f))=0$, the Moebius transformation $T$ should satisfy $T(0)=z_0$. Moreover, it is required that $T$ maps the disk $r\UD$, $r:=e^{t_0-t_1}$, onto itself. Finally, $T$ should have no pole in~$\overline\UD$. It is easy to check that
\begin{equation}\label{EQ_T}
  T(z):=(1+|z_0|^2)\frac{z+z_0}{1+|z_0|^2+2\bar z_0 z},\quad z\in\UD,
\end{equation}
satisfies the above three requirements with $r:=\sqrt{(1+|z_0|^2)/2}$, which corresponds to
\begin{equation}\label{EQ_t_1}
  t_1=t_0+\tfrac12\log(2/(1+|z_0|^2)).
\end{equation}

Denote $\eta(z):=\log(T^{-1}(z))$, $z\in\UD\setminus\{z_0\}$. This function itself is multivalued, but $\Re \eta$ and $\eta'$ are single-valued in $\UD\setminus\{z_0\}$.  Calculate the Laplacian of ${\varphi(w)=\exp(-a\Re \eta(H(w)))}$, where $H:=F^{-1}$ and $F$ is defined in Lemma~\ref{LM_dF}, i.e. $F(z):=G(e^{-t_0}z)$, $z\in\UD$. Denote $u(w):=\Re \eta(H(w))$. Then
$$
\Delta\varphi(w)=a\varphi(w)\big(a|\nabla u(w)|^2-\Delta u(w)\big),\quad w\in D_{t_0}\setminus\{-a_2(f)\}.
$$
Therefore, in order to prove that there exists $a>0$ depending only on~$q\in(0,1/3)$ such that $\varphi$ is subharmonic in~$D_{t_0}$, we have to show that $\Delta u(w)/|\nabla u(w)|^2$ has an upper bound depending only on~$q$.

Note that $u(w)$ is real-valued. Regarding the vector $\nabla u(w)$ as a complex number, we have $\nabla u=2\bar\partial u=\overline{(\eta'\circ H)\partial H}+(\eta'\circ H)\bar\partial H$. By Becker's result~\cite[Sect.\,5.2]{Becker80}, $G$ is a $k$-q.c. map with $k:=3q$. (This can be seen also directly by calculating $|\bar\partial G|/|\partial G|$.) Therefore, $H$ is also $k$-q.c. and hence
\begin{equation}\label{EQ_gradient-est}
  \big|\nabla u(w)\big|\ge\big|\,|\overline{\eta'(H(w))\partial H(w)}|\,-\,|\eta'(H(w))\bar\partial H(w)|\,\big|\ge(1-k)\big|\eta'(H(w))\partial H(w)\big|
\end{equation}
for all $w\in D_{t_0}\setminus\{-a_2(f)\}$. Moreover,
$$
 \Delta u(w)=4\partial\bar\partial u(w)=\Re\Big(\eta'(H(w))\Delta H(w)+4\eta''(H(w))\partial H(w)\bar\partial H(w)\Big).
$$
Therefore,
$$
\frac{|\Delta u(w)|}{|\nabla u(w)|^2}\le\frac1{(1-k)^2}\left(\left|\frac{\Delta H(w)}{\eta'(H(w))\partial H(w)^2}\right|+4\left|\frac{\eta''(H(w))\bar \partial H(w)}{\eta'(H(w))^2\partial H(w)}\right|\right).
$$
The second term in the r.h.s. is easy to estimate. Indeed, $|\bar\partial H|\le k |\partial H|$ because $H$ is a $k$-q.c. map. Moreover,
$$
\frac{\eta''(z)}{\eta'(z)^2}=\frac{4\bar z_0z-(1+3|z_0|^2)}{1-|z_0|^2}.
$$
Hence $|\eta''(z)|/|\eta'(z)|^2\le 8/(1-|z_0|^2)\le 8/(1-k)$ for all $z\in\UD$.

To estimate the first term, we notice that $\eta'(z)=(1-|z_0|^2)(z-z_0)^{-1}(1+|z_0|^2-2\bar z_0z)^{-1}.$ It follows that
$$
|\eta'(z)|\ge\frac{1-|z_0|}{(1+|z_0|)^2}\ge\frac{1-\sqrt{k}}{(1+\sqrt{k})^2}\quad \text{for all $z\in\UD\setminus\{z_0\}$.}
$$
Furthermore, $|\partial H(w)|^2\ge |\partial H(w)|^2 - |\bar\partial H(w)|^2=J_H(w)=1/J_F(H(w)),$ which is greater or equal to~$k^{-1}\,(1-k^2)^4\,(1-k)^{2q}$ by Lemma~\ref{LM_dF}. Therefore, it remains to show that $|\Delta H(w)|$ has an upper bound in~$D_{t_0}$ depending only on~$q$.

Denote $U(z):=-\bar\partial F(z)/J_F(z)$. Then
\begin{equation}\label{EQ_partial-U}
\partial U=\displaystyle-\frac{\partial\bar\partial F}{J_F^2}+2\bar\partial F\frac{\partial J_F}{J_F^3},\quad
\bar\partial U=\displaystyle-\frac{\bar\partial^2 F}{J_F^2}+2\bar\partial F\frac{\bar\partial J}{J_F^3}.
\end{equation}
Taking into account that $J_F=|\partial F|^2-|\bar\partial F|^2\ge(1-k^2)|\partial F|^2$, with the help of Lemma~\ref{LM_est-second-formal} we see that for all $z\in\UD$,
\begin{equation}\label{EQ_est-partial-U}
  \max\big\{|\partial U|,|\bar\partial U|\big\}\le M_1(q)\,|\partial F|^{-3},\quad M_1(q):=\frac{M(q)}{(1-k^2)^2}\Big(1+\frac{8}{1-k^2}\Big).
\end{equation}
Similarly, we get
\begin{equation}\label{EQ_est-for-U-V}
|U(z)|\le\frac{k}{1-k^2}\,|\partial F|^{-1},\quad|V(z)|\le\frac{1}{1-k^2}\,|\partial F|^{-1}
\end{equation}
for all $z\in\UD$, where $V(z):=\overline{\partial F(z)}/J_F(z)$.

Since $\bar\partial H=U\circ H$ and $\partial H=V\circ H$, for all $w\in D_{t_0}$ we have
\begin{multline*}
  \frac{\Delta H(w)}{4}=\partial\bar\partial H(w)=\partial U(z)\partial H(w)+\bar\partial U(z)\overline{\bar\partial H(w)}\\=\partial U(z)\,V(z)\,+\,\bar\partial U(z)\,\overline{U(z)},~\text{~where~}z:=H(w).
\end{multline*}
Using \eqref{EQ_est-partial-U} and~\eqref{EQ_est-for-U-V}, we see that $|\Delta H(w)|\le M_1(q)(1-k)^{-1}|\partial F(z)|^{-4}$ for all ${w\in D_{t_0}}$.
To complete the proof, it is remains to apply the lower estimate for $|\partial F|$ given in Lemma~\ref{LM_dF}.\qed

\subsection{Estimates for functions in $\Sigma(k)$ with smooth image domains.}
In what follows, $\rc(z_0,D)$ will stand for the conformal radius of a domain $D\subset\Complex$ w.r.t. the point $z_0\in D$. For a point $z\in\C$ and two sets $A,B\subset\Complex$, we define
$$
 \dist(z,B):=\inf\{|z-w|:w\in B\},\quad
 \dist(A,B):=\inf\{|z-w|:z\in A,w\in B\}.
$$
Furthermore, for a map $g:\Delta\to\ComplexE$ and $d>0$, we denote $$\mathcal A_g(d):=\{w\in g(\Delta):\dist(w,\partial g(\Delta))\le d\}.$$

According to a well-known result by K\"uhnau~\cite[Satz~4]{Kuhnau69}, for any $g\in\Sgk$,
\begin{equation}\label{EQ_Kuhnau-der}
(1-|z|^{-2})^k\le|g'(z)|\le\frac{1}{(1-|z|^{-2})^k},\qquad \mathrlap{z\in\Delta.}
\end{equation}

\begin{proposition}\label{PR_dist}
  Let $R>1$ and $k\in(0,1)$. Then for any $g\in\Sgk$,
  $$
  \mathcal A_g\big(d_1(k,R)\big)\subset g\big(\{z:1<|z|\le R\}\big)\subset \mathcal A_g\big(d_2(k,R)\big),
  $$
  where $d_1(k,R):=\tfrac14R^{-2k}(R-1)^{1+k}(R+1)^k$ and $d_2(k,R):=4R^{2k}(R-1)^{1-k}(R+1)^{-k}.$
\end{proposition}
\begin{proof}
Let us fix some $z_0\in\Delta$. Denote by $D$ the open disk of radius $\dist(g(z_0),\partial g(\Delta))$ centered at~$g(z_0)$ and let $\Omega:=g^{-1}(D)$. Then $\dist(z_0,\partial \Omega)\le |z_0|-1$. Therefore, with the help of the upper bound in~\eqref{EQ_Kuhnau-der} and Koebe's 1/4-Theorem we obtain:
\begin{multline}\label{EQ_dist2}
 \dist(g(z_0),\partial g(\Delta))=\\=\rc(g(z_0),D)=|g'(z_0)|\,\rc(z_0,\Omega)\le4|g'(z_0)|\,\dist(z_0,\partial \Omega)\le d_2(k,|z_0|).
\end{multline}

Similarly, denoting $\Omega':=\{z:|z-z_0|<|z_0|-1\}\subset\Delta$ and $D':=g(\Omega')$, we get:
\begin{multline}\label{EQ_dist1}
  \dist(g(z_0),\partial g(\Delta))\ge \dist(g(z_0),\partial D')\ge \tfrac14 \rc(g(z_0),D')=\\=\tfrac14|g'(z_0)|\,\rc(z_0,\Omega')\ge d_1(k,|z_0|).
\end{multline}

 To complete the proof, simply apply inequalities~\eqref{EQ_dist2} and~\eqref{EQ_dist1} with $1<|z_0|\le R$ and take into account that $d_1(k,\cdot)$ and $d_2(k,\cdot)$ are strictly increasing on~$[1,+\infty)$.
\end{proof}

\begin{remark}
According to Proposition~\ref{PR_dist}, the preimage of
any $\varepsilon$-neighborhood of $\partial g(\Delta)$ contains an annulus $\{z:1<|z|<R\}$, where $R>1$ depends on $\varepsilon$ and~$k$, but not on the choice of~$g\in\Sgk$. Note that this property does not hold for the whole class~$\Sigma$. Indeed, for any $\delta\in(0,\pi)$ and suitable $r_\delta>1$ there is a unique $g_\varepsilon\in\Sigma$ that maps $\Delta$ onto $D_\delta:=\ComplexE\setminus\{r_\delta e^{i\theta}:|\theta|\le\pi-\delta\}$. Using the Carath\'eodory Kernel Convergence Theorem, see e.g. \cite[\S3.1]{Duren}, we see that $g_\delta(z)\to z$ locally uniformly in~$\Delta$ and hence $|g^{-1}_\delta(0)|\to 1$  as ${\delta\to0^+}$, although $\dist(0,\partial g_\delta(\Delta))=r_\delta>1$ for all $\delta\in(0,\pi)$.
\end{remark}

The following lemma can be used to estimate the derivative of the Riemann map on the unit circle via the geometric quantities describing the image domain. This idea is borrowed from~\cite{Range1976}.
\begin{lemma}\label{LM_Henkin}
  Let $g$ be a conformal map of $\Delta$ onto a domain $\Omega\ni\infty$ bounded by a $C^2$-smooth Jordan curve. Let  $u:U\to\Real$ be a $C^1$-smooth function on a neighbourhood~$U$ of~$\partial\Omega$. Suppose that $u$ vanishes on~$\partial \Omega$ and that it is negative and subharmonic in~$U\cap\Omega$. If the image of $\{z:1<|z|\le R\}$ w.r.t. $g$ lies in~$U$ for some $R>1$, then
  \begin{equation}\label{EQ_Henkin}
    |g'(z)|\ge-\frac{4 u_0}{\pi(R-1)|\nabla u(g(z))|}\quad\text{for all~$z\in\partial\Delta$},
  \end{equation}
  where $u_0:=\max_{z\in A_0}u(g(z))<0$ and $A_0:=\{z:\sqrt{(1+R^2)/2}\le|z|\le R\}$.
\end{lemma}
\begin{proof}
By the hypothesis, $\partial \Omega$ is $C^2$-smooth. It follows that $g'$, and hence the gradient of $v:=u\circ g$, extend continuously to the unit circle, see e.g. \cite[Theorem~3.5 on p.\,48]{Pommerenke:BB}.

Fix some $\alpha\in[0,2\pi]$. Since $v$ is subharmonic in~$A:=\{z:1<|z|<R\}$ and since it is continuous and non-positive on the closure of~$A$, for any $\rho\in(1,R)$ we have:
$$
 v(\rho e^{i\alpha})\le \frac{1}{2\pi}\int_0^{2\pi}\mathcal P_\alpha(\rho e^{i\alpha},\theta)\, v\big(z_0+r e^{i\theta}\big)\,\di\theta,
$$
where $$\mathcal P_\alpha(z,\theta):=\Re\frac{re^{i\theta}+z-z_0}{re^{i\theta}-(z-z_0)},\quad z_0:=\frac{R+1}2e^{i\alpha},~r:=\frac{R-1}2,$$
is the Poisson kernel for the disk $D_\alpha:=\{z:|z-z_0|<r\}$. Note that the intersection of~$A$ with the ray $\{te^{i\alpha}:t\ge0\}$ is a diameter of~$D_\alpha$ and that exactly one half of the circle~$\partial D_\alpha$  lies in~$A_0$. Therefore, denoting $I_{\alpha}:=\big\{\theta\in[0,2\pi]:z_0+r e^{i\theta}\in A_0\big\}$, we get:
\begin{multline}\label{EQ_v-est}
v(\rho e^{i\alpha})\,\le\,
\frac{u_0}{2\pi}\int_{I_{\alpha}\!}\mathcal P_\alpha(\rho e^{i\alpha},\theta)\,\di\theta\,
=\,\frac{u_0}{2\pi}\int\limits_{\pi/2}^{3\pi/2}\Re \frac{e^{i\vartheta}+x}{e^{i\vartheta}-x}\,\di\vartheta~=\\=~\frac{u_0}{2\pi}\int\limits_{\pi/2}^{3\pi/2} \frac{e^{i\vartheta}+x}{e^{i\vartheta}-x}\,\di\vartheta~=\,u_0\Big(\frac12-\frac2\pi\arctg x\Big),\quad x:=1-\frac{\rho-1}{r}.
\end{multline}
Recall that $v$ vanishes on the unit circle. Therefore, using~\eqref{EQ_v-est} we obtain
$$
|\nabla v(e^{i\alpha})|=-\left.\frac{\partial v(\rho e^{i\alpha})}{\partial\rho}\right|_{\rho=1}\ge\,-\frac{4u_0}{\pi r}.
$$
It remains to notice that $|g'(e^{i\alpha})|=|\nabla v(e^{i\alpha})|/|\nabla u(g(e^{i\alpha}))|$.
\end{proof}
\begin{corollary}\label{CR_Henkin}
  In conditions of Lemma~\ref{LM_Henkin}, suppose that $g\in\Sgk$ and that $\mathcal A_g(\mu)\subset U$ for some $\mu\in(0,8]$. Then
\begin{equation}\label{EQ_Henkin1}
    |g'(z)|\ge\frac{4\inf\big\{|u(w)|:w\in U\cap\mathcal B_g\big(\alpha(k)\mu^K\big)\big\}}{\pi(\mu/8)^{1/(1-k)}|\nabla u(g(z))|}\quad\text{for all~$z\in\partial\Delta$},
  \end{equation}
where $\mathcal B_g(d):=g(\Delta)\setminus\mathcal A_g(d)$, $K:=(1+k)/(1-k)$, and $\alpha(k):=1/\big(3^{1+k}\cdot8^{K+1}\big)$.
\end{corollary}
\begin{proof}
We apply Lemma~\ref{LM_Henkin} with $R:=1+(\mu/8)^{1/(1-k)}\le 2$. To ensure that $$g\big(\{z:1<|z|\le R\}\big)\subset U,$$ we use Proposition~\ref{PR_dist} together with the elementary estimate
$$
d_2(k,R)=4\Big(\frac{R^2}{R+1}\Big)^k(R-1)^{1-k}\le 4\cdot2^k(R-1)^{1-k}<\mu.
$$
Similarly, in view of other two elementary estimates:
\begin{gather*}
  2>R_*:=\sqrt{\frac{R^2+1}{2}}>\sqrt{R}\ge 1+(\sqrt2-1)(R-1)\ge 1+\frac{(\mu/8)^{1/(1-k)}}3,\\[.75ex]
  d_1(k,R_*)=\frac14\Big(\dfrac{R_*+1}{R_*^2}\Big)^k(R_*-1)^{1+k}>\frac{(R_*-1)^{1+k}}{4\cdot 2^k}>\frac{(\mu/8)^K/3^{1+k}}{8}=\alpha(k)\mu^K,
\end{gather*}
Proposition~\ref{PR_dist} implies that $g(A_0)\subset U\cap\mathcal B_g\big(\alpha(k)\mu^K\big)$. Hence~\eqref{EQ_Henkin1} follows from~\eqref{EQ_Henkin}.
\end{proof}

We use a somewhat similar argument to estimate $|g'(z)|$ on $\partial\Delta$ from above. Note that in this case, we actually do not need to assume that the boundary is smooth.
\begin{proposition}\label{PR_derivative-from-ABOVE}
 There exists  $\mathcal M:(0,+\infty)\times(0,1)\to(0,+\infty)$ such that for any ${\varepsilon>0}$ and any ${k\in(0,1)}$, the following assertion holds for all $g\in\Sgk$: if $g(\Delta)$ contains an open disk~$D$ of radius~$\varepsilon$ such that $\partial D$ and $\partial g(\Delta)$ have a common point~$w_0$, then
\begin{equation}\label{EQ_mathcal-M}
  \big|\angle g'(\zeta_0)\big|\le \mathcal M(\varepsilon,k),
\end{equation}
where $\angle g'(\zeta_0)$ stands for the angular derivative of $g$ at the landing point~$\zeta_0$ of the slit~$g^{-1}(I)$ and $I$ is the straight line segment joining $w_0$ with the center of~$D$.
\end{proposition}
\begin{proof}
The fact that preimages of slits are slits is well-known, see e.g. \cite[Theorem 9.2 on p.\,267]{Pommerenke}.
Existence of finite angular derivative follows from \cite[Theorem 10.6 on p.\,307]{Pommerenke}.

If $\angle g'(\zeta_0)=0$, then there is nothing to prove. So suppose that $\angle g'(\zeta_0)\neq0$. In such a case, the slit $\gamma:=g^{-1}(I)$ tends to~$\partial\Delta$ non-tangentially. Hence, $\angle g'(\zeta_0)=\lim_{\gamma\ni z\to\zeta_0}g'(z)$. Therefore, it is sufficient to find an upper bound of $|g'(z)|$ for $z\in\gamma$ close to~$\zeta_0$.

Using translations and rotations, we may suppose that $D=\{w:|w|<\varepsilon\}$ and ${w_0=\varepsilon}$.
Consider the function $u(w):=\log|g^{-1}(w)|$. It is harmonic and positive in~$D$. Moreover, it extends continuously to~$\partial D$, with $u(w_0)=0$. Fix $w\in(0,\varepsilon)$ and apply the Poisson representation for $u$ in the smaller disk $D_1:=\{w:|w-\varepsilon/2|<\varepsilon/2\}$. We have
\begin{multline}\label{EQ_u-harmonic-est}
u(w)=\frac{1}{2\pi}\int_0^{2\pi}\mathcal P_1(w,\theta)\, u\big((1+e^{i\theta})\varepsilon/2\big)\,\di\theta~\ge~ \frac{u_0}{2\pi}\int_{\pi/2}^{3\pi/2}\mathcal P_1(w,\theta)\,\di\theta\,=\\= \,u_0\Big(\frac12-\frac2\pi\arctg \Big(\frac{2w}{\varepsilon}-1\Big)\Big),~~u_0:=\min_{\zeta\in C}u(\zeta),
\end{multline}
where $\mathcal P_1$ is the Poisson kernel in~$D_1$ and $C:=\{\zeta\in\partial D_1:\Re\zeta\le\varepsilon/2\}$.

We have $\dist(w,\partial g(\Delta))=\varepsilon-w$, while
$$
\dist(g^{-1}(w),\partial\Delta)=\exp(u(w))-1\ge u(w)\ge\frac{2u_0}{\pi\varepsilon}(\varepsilon-w)+o(\varepsilon-w)\quad\text{as $~w\to\varepsilon$}.
$$
As in the proof of Proposition~\ref{PR_dist}, we have
$$
|g'\big(g^{-1}(w)\big)|\le \frac{4\,\dist(w,\partial g(\Delta))}{\dist(g^{-1}(w),\partial\Delta)}\le\frac{2\pi\varepsilon}{u_0}\big(1+o(\varepsilon-w)\big).
$$

It remains to estimate~$u_0$. To this end we use Proposition~\ref{PR_dist}. By construction, $\dist(C,\partial g(\Delta))\ge\varepsilon_1:=(1-1/\sqrt2)\varepsilon$. Hence, $u_0\ge\log R(\varepsilon,k)$, where $R=R(\varepsilon,k)>1$ is the unique solution to the equation $d_2(k,R)=\varepsilon_1$. Existence and uniqueness of the solution follows form the fact that for any fixed $k\in(0,1)$, $R\mapsto d_2(k,R)$ is a strictly increasing map of $[1,+\infty)$ onto $[0,+\infty)$.
Since $R$ depends only on $k$ and $\varepsilon$, the proposition is now proved.
\end{proof}

\begin{remark}
Note that the upper estimate $\mathcal M(\varepsilon,k)=2\pi\varepsilon/\log R(\varepsilon,k)$ for $|\angle g'|$ obtained above explode to~$+\infty$ both as $\varepsilon\to +0$ and as $\varepsilon\to +\infty$. In particular, there exists certain~$\varepsilon^*>0$ for which $\mathcal M(\varepsilon,k)$ takes its minimal value. This provides an upper bound, depending only on~$k$, for the angular derivatives of functions $g\in\Sgk$ such that $\Complex\setminus g(\Delta)$ is convex. It is curious enough to mention that the latter bound does not explode as $k\to1$ and in fact, gives an absolute bound $|\angle g'|<165$ for any $g\in\Sigma$ with convex~ ${\Complex\setminus g(\Delta)}$.
\end{remark}

\subsection{Distance, diameter, and curvature estimates}\label{SS_dist-diam}
Recall that we defined the choice of $t_1>t_0$ and of the Moebius transformation~$T$ at the beginning of Sect.\,\ref{SS_proof-of-proposition-subharmonic}. Moreover, we choose the largest $t_2\in(t_0,t_1)$ such that $T(e^{t_2-t_0}\UD)\subset\UD$. Namely,
\begin{equation}\label{EQ_t_2}
t_2=t_0+\log\big((1+|z_0|)/(1+|z_0|^2)\big),
\end{equation}
 where as before $z_0:=e^{t_0}G^{-1}(0)$.
Recall also that in Lemma~\ref{LM_dF} we defined ${F(z):=G(e^{-t_0}z)}$ for all~$z\in\UD$.
Finally, let $\Gamma_t:=\partial g_t(\Delta)$.

In addition to $\dist(A,B)$, for two sets $A,B\subset\C$ we define
$$
 \dist^*(A,B):=\max\big\{\sup_{z\in A}\dist(z,B),\,\sup_{w\in B}\dist(w,A)\big\}.
$$
\begin{lemma}\label{LM_Loewner-front-distance}
  In the above notation, there exist positive constants $M_*(q)$ and $M^*(q)$ depending only on~$q$ such that
\begin{equation}\label{EQ_dist-Dt0-Dt1}
  M_*(q)\le \dist(\Gamma_{t_0},\Gamma_{t_1})\le\dist^*(\Gamma_{t_0},\Gamma_{t_1}) \le M^*(q).
\end{equation}
Moreover, for any $t>t_2$ and any $s\in[t_2,t)$, we have
  \begin{equation}\label{EQ_Loewner-dist-est}
    \dist^*\big(L(\Gamma_t),L(\Gamma_s)\big)\le M_2(q)(e^{-s}-e^{-t})
  \end{equation}
for some constant $M_2(q)$ depending only on~$q$.
\end{lemma}
\begin{proof}
For a $C^1$-map $V:D\to\C$ of an open set $D\subset\C$ and $\alpha\in\Real$, we denote:
$$
\mathsf D_\alpha V(z):=\tfrac{\di}{\di s} V(z+se^{i\alpha})|_{s=0},\quad \mathsf D^*V(z):=\max_{\alpha\in\Real}\big|\mathsf D_\alpha V(z)\big|,\quad \mathsf D_*V(z):=\min_{\alpha\in\Real}\big|\mathsf D_\alpha V(z)\big|.
$$

Note that $F$ extends diffeomorphically to a disk larger than $\UD$, namely to $e^{t_0}\UD$. The curves $\Gamma_{t_0}$ and $\Gamma_{t_1}$ are images under~$F$ of $A:=\{z:|z|=1\}$ and $B:=\{z:|z|=e^{t_0-t_1}\}$, respectively.  Fix a point $z\in A$ and let $\gamma$ be the straight line segment joining $z$ with the closest point of $B$. Then the length of $F(\gamma)$ does not exceed
$(1-e^{t_0-t_1})\max\mathsf D^*F(\zeta)$, where the maximum is taken over all $\zeta\in\gamma$. Applying this simple argument again, but with $A$ and $B$ swapped, we conclude that
$$
 \dist^*\big(\Gamma_{t_0},\Gamma_{t_1}\big)\le(1-e^{t_0-t_1})\max_{\zeta\in\overline\UD}\mathsf D^*F(z)<\max_{\zeta\in\overline\UD}\mathsf D^*F(z).
$$
Hence, the upper bound in~\eqref{EQ_dist-Dt0-Dt1} follows directly from Lemma~\ref{LM_dF}.

\medskip

In a similar way, we prove~\eqref{EQ_Loewner-dist-est}.  The curves $L(\Gamma_t)$ and $L(\Gamma_s)$ are images under the map $L\circ F=F\circ T$ of $A:=\{z:|z|\le e^{t_0-t}\}$ and $B:=\{z:|z|\le e^{t_0-s}\}$, respectively.
Clearly, $\mathsf D^*(F\circ T)(z)=\mathsf D^*F(T(z))|T'(z)|$.
Since $s,t\ge t_2$, both $T(A)$ and $T(B)$ lie in~$\overline\UD$. Therefore, as above, we can use the upper for ${\mathsf D^* F(T(z))}$ given in Lemma~\ref{LM_dF}, and
\begin{multline}\label{EQ_T-prime-from-above}
|T'(z)|=\frac{1-|z_0|^4}{\big|1+|z_0|^2+2\bar z_0 z\big|^2}\le\frac{1-|z_0|^4}{\big(1+|z_0|^2\,-\,2|z_0|\,|z|\big)^2}=\\=(1+|z_0|^2)\frac{1+|z_0|}{(1-|z_0|)^3}\le (1+3q)\frac{1+\sqrt{3q}}{(1-\sqrt{3q})^3}
\end{multline}
for all~$z\in\overline\UD$. Now \eqref{EQ_Loewner-dist-est} follows easily with $$M_2(q):=1\big/\big((1-k)^{1+q}(1-\sqrt{k})^4\big),\quad\text{ where $k:=3q$.}$$

It remains to prove the lower estimate in~\eqref{EQ_dist-Dt0-Dt1}. Since $\Gamma_{t_0}$ and $\Gamma_{t_1}$ are two nested Jordan curves, it is not difficult to see that there exists a straight line segment~$I\subset E$ joining $\Gamma_{t_0}$ with $\Gamma_{t_1}$ whose length is exactly $\dist(\Gamma_{t_0},\Gamma_{t_1})$. Here $E$ stands for the closure of the doubly connected domain bounded by $\Gamma_{t_0}$ and $\Gamma_{t_1}$. The length of $\gamma:=F^{-1}(I)$ is at least $1-e^{t_0-t_1}$. Therefore,
$$
1-e^{t_0-t_1}\le\dist(\Gamma_{t_0},\Gamma_{t_1})\max_{w\in I}\mathsf D^*(F^{-1})(w)=\frac{\dist(\Gamma_{t_0},\Gamma_{t_1})}{\min_{z\in \gamma}\mathsf D_*F(z)}.
$$
To complete the proof it is sufficient to use the lower bound in~\eqref{EQ_DF}  and take into account that $t_1-t_0\ge\tfrac12\log\big(2/(1+3q)\big)$, see the proof of Proposition~\ref{PR_subharmonic} in Sect.\,\ref{SS_proof-of-proposition-subharmonic}.
\end{proof}
\begin{lemma}\label{LM_diam}
For any $t>t_2$,
\begin{equation}\label{EQ_diam}
 M_3(q)e^{-t}\le\diam(L(\Gamma_t))\le M_4(q)e^{-t},
\end{equation}
where  $M_3(q)$ and $M_4(q)$ are positive constants depending only on~$q$.
\end{lemma}
\begin{proof}
The upper bound in~\eqref{EQ_diam} can be obtained using the same method as in the proof of Lemma~\ref{LM_Loewner-front-distance}.
Let us obtain the lower bound in~\eqref{EQ_diam}. Denote $A:=\{z:|z|=e^{t_0-t}\}$. By a direct computation using~\eqref{EQ_T}, we find that
 $B:=T(A)\subset\UD$ is a circle of radius
\begin{equation*}
e^{t_0-t}\,\frac{1-|z_0|^4}{1+|z_0|^4+(2-4e^{2(t_0-t)})|z_0|^2}\ge e^{t_0-t}\,\frac{1-|z_0|^2}{1+|z_0|^2}\ge e^{t_0-t}\,\frac{1-3q}{1+3q},
\end{equation*}
where the last inequality is due to~\eqref{EQ_z_0-est}. 

Hence, $\diam B\ge 2e^{t_0-t}(1-k)/(1+k)$, where $k:=3q$, and our task reduces to finding a lower bound for $\diam(F(B))/\diam(B)$. The technique used to in Lemma~\ref{LM_Loewner-front-distance} does not apply directly to this case, because the longest straight line segment with the end-points in~$F(B)=L(\Gamma_t)$ does not have to lie in~$F(\UD)=D_{t_0}$. However, $D_{t_0}$ is a $k$-quasidisk and hence by a result of Gehring and Osgood, see e.g. \cite[\S8.1]{KariHag}, for any two points $w_1,w_2\in D_{t_0}$ there is a smooth curve $\gamma\subset D_{t_0}$ joining $w_1$ and $w_2$ such that $\len(\gamma)\le c(k)|w_2-w_1|$, where $c(k)>0$ is a constant depending only on~$k$. At the same time
$$
|F^{-1}(w_2)-F^{-1}(w_1)|\le\len(F^{-1}(\gamma))\le\frac{\len(\gamma)}{\inf_{z\in\UD}\mathsf D_*F(z)}.
$$
By choosing $w_j:=F(z_j)$, $j=1,2$, where $z_1$ and $z_2$ are the end-point of a diameter of~$B$, we obtain
\begin{eqnarray*}
  \diam(F(B))&\ge& |w_2-w_1|~\ge~\frac{\len(\gamma)}{c(k)}\\[1ex]&\ge&\frac{|z_2-z_1|\inf_{z\in\UD}\mathsf D_*F(z)}{c(k)}=\diam(B)\frac{\inf_{z\in\UD}\mathsf D_*F(z)}{c(k)}.
\end{eqnarray*}
Employing the lower estimate for $\mathsf D_*F$ given in Lemma~\ref{LM_dF} completes the proof.
\end{proof}

\begin{lemma}\label{LM_curvature}
There exist positive constants $M_5(q)$ and $M_6(q)$ depending only on~$q$ such that for all $t\ge t_2$
the curvature of $L(\Gamma_t)$ does not exceed $\kappa_0(q,t):=M_5(q)e^{t}+M_6(q)$.
\end{lemma}
\begin{proof}
Denote $A:=\{z:|z|=e^{t_0-t}\}$. Then $L(\Gamma_t)=F(T(A))$.
As in the proof of Lemma~\ref{LM_diam}, we see that
 $T(A)\subset\UD$ is a circle of radius
\begin{equation}\label{EQ_radius}
\rho\ge e^{t_0-t}\frac{1-k}{1+k},\quad k:=3q.
\end{equation}

Parameterize $T(A)$ by $z(\theta)=z_*+\rho e^{i\theta}$, $\theta\in[0,2\pi]$.  Then $w(\theta)=F(z(\theta))$ is a parametrization of $L(\Gamma_t)$. The (signed) curvature of $L(\Gamma_t)$ at the point $w=w(\theta)$ equals
$$
\kappa(w)=\frac1{\big|w'(\theta)\big|}\,\Im\frac{\,w''(\theta)\,}{w'(\theta)}.
$$
By Lemma~\ref{LM_dF},
\begin{equation}\label{EQ_w-prime}
  |w'(\theta)|\ge |z'(\theta)|\,\mathsf D_* F(z(\theta))\ge \rho e^{-t_0}\frac{(1-k)^{1+q}}{(1+k^2)^2},\quad k:=3q.
\end{equation}

Moreover,
$$
\frac{\,w''(\theta)\,}{w'(\theta)}=i\,\frac{e^{i\theta}\partial F+e^{-i\theta}\bar\partial F+\rho\big(e^{2i\theta}\partial^2F -2\partial\bar\partial F +e^{-2i\theta}\bar\partial^2F\big)}%
{\vphantom{\int_0^1}e^{i\theta}\partial F-e^{-i\theta}\bar\partial F},
$$
where the partial derivatives of $F$ are to be evaluated at~$z(\theta)$. Taking into account that $F$ is $k$-q.c. and using Lemma~\ref{LM_est-second-formal},
we see that $$\big|{w''(\theta)}/{w'(\theta)}\big|\le K+\frac{4M(q)\rho}{1-k},\quad K:=\frac{1+k}{1-k}.$$

Combining the above inequality with~\eqref{EQ_radius} and~\eqref{EQ_w-prime} and taking into account that $k$ and $M(q)$ depend only on~$q$ leads to the desired conclusion.
\end{proof}

\begin{lemma}\label{LM_krug}
  Let $\Gamma$ be a convex $C^2$-smooth Jordan curve and  $\kappa_*$ the maximum of its curvature. If $C$ is a circle of radius~$R$ not exceeding $1/\kappa_*$ internally tangent to~$\Gamma$, then the open disk~$D$ bounded by~$C$ is a subset of the Jordan domain $\Omega$ bounded by~$\Gamma$.
\end{lemma}
\begin{proof}
  Clearly, if we prove the statement of Lemma~\ref{LM_krug} for all $0<R<1/\kappa_*$, then it also holds with $R=1/\kappa_*$. Suppose that it fails for some $R<1/\kappa_*$. Then there exists $R_0\in(0,1/\kappa_*)$ and a circle $C_0$ of radius $R_0$ which is internally tangent to~$\Gamma$ at two distinct points $z_1$ and $z_2$. Using if necessary translations, rotations, and/or reflections, we may suppose that $C_0$ is centered at the origin and that $z_1=R_0e^{-i\theta}$ and $z_2=R_0e^{i\theta}$ for some $\theta\in(0,\pi/2]$.

Let $\Gamma_0$ be the arc of~$\Gamma$ going in the counterclockwise direction from~$z_1$ to~$z_2$. Let $z=\varphi(t)$, $t\in[-1,1]$, be a $C^2$-parametrization of~$\Gamma_0$, with $z_1=\varphi(-1)$ and $z_2=\varphi(1)$. Then the $C^1$-function $\beta(t):=\arg\varphi'(t)$ is non-decreasing, because $\Gamma$ is convex, and $\beta(\pm1)=\pi/2\pm\theta$ because $\Gamma$ is tangent to~$C$ at $z_1$ and~$z_2$. It follows that $0\le\beta(t)\le \pi$ for all $t\in[-1,1]$.

Note also that  $\beta'(t)\le \kappa(t)|\varphi'(t)|\le \kappa_*|\varphi'(t)|$, where $\kappa(t)$ stands for the curvature of~$\Gamma$ at~$z=\varphi(t)$. Therefore,
\begin{multline*}
2R_0\sin\theta=\Im(z_2-z_1)=\int_{-1}^{1}|\varphi'(t)|\sin\beta(t)\,\di t\ge\\\ge \int_{-1}^{1}\kappa_*^{-1}\beta'(t)\sin\beta(t)\,\di t=\kappa_*^{-1}\big(\cos\beta(-1)-\cos\beta(1)\big)> 2R_0\sin\theta.
\end{multline*}
The contradiction we have obtained proves the lemma.
\end{proof}

\begin{lemma}\label{LM_tangent-disk}
 There exists a constant $\varepsilon_0(q)>0$ depending only on~$q$ such that for any $t>t_1$ and any $w_0\in L(\Gamma_t)$ the open disk of radius $\varepsilon_0(q)$ with the boundary tangent to $L(\Gamma_t)$ at~$w_0$ from outside is entirely contained in the unbounded component of $\C\setminus L(\Gamma_t)$.
\end{lemma}
\begin{proof}
Fix $t>t_1$. Recall that $L\circ F=F\circ T$, with $T$ mapping $e^{t_0-t_1}\UD$ onto itself. Therefore, $L(\Gamma_t)$ is the image of a circle~$C(t)\subset e^{t_0-t_1}\UD$ w.r.t. $F$. It follows that for any $\rho\in(0,1-e^{t_0-t_1}]$ and any point $z\in C(t)$, the unbounded component of $\C\setminus L(\Gamma_t)$ contains a smooth Jordan domain~$U_t(z,\rho)$ with $L(\Gamma_t)$ and $\partial U_t(z,\rho)$ tangent at the point~$F(z)$ and such that $F^{-1}(U_t(z,\rho))$ is a disk of radius~$\rho$ lying together with its boundary in~$\UD$.

Since $F$ is a diffeomorphic $k$-q.c. map,
$$
|\bar\partial F(z)|\le k|\partial F(z)|\quad\text{and}\quad\Re\frac{\,e^{i\theta}\partial F(z)+e^{-i\theta}\bar\partial F(z)\,}{\vphantom{\int_0^1}e^{i\theta}\partial F(z)-e^{-i\theta}\bar\partial F(z)}\ge \frac{1-k}{1+k}
$$
for all $z\in\UD$. Arguing as in the proof of Lemma~\ref{LM_curvature}, we see that the domain~$U_t(z,\rho)$ is convex whenever
$$
\Re\frac{e^{i\theta}\partial F+e^{-i\theta}\bar\partial F}%
{\vphantom{\int_0^1}e^{i\theta}\partial F-e^{-i\theta}\bar\partial F}~\ge~
\left|\frac{\rho\big(e^{2i\theta}\partial^2F -2\partial\bar\partial F +e^{-2i\theta}\bar\partial^2F\big)}%
{\vphantom{\int_0^1}e^{i\theta}\partial F-e^{-i\theta}\bar\partial F}\right|.
$$
Thanks to Lemma~\ref{LM_est-second-formal}, the latter inequality holds provided
$$
\rho\le \rho_1(q):=\frac{(1-k)^2}{4(1+k) M(q)}.
$$

Note that by~\eqref{EQ_t_1} and~\eqref{EQ_z_0-est}, we have $1-e^{t_0-t_1}\ge 1-\sqrt{(1+3q)/2}=:\rho_2(q)>0$.
Now set $\rho=\rho_0(q):=\min\{\rho_1(q),\rho_2(q)\}>0$. Using again the proof of Lemma~\ref{LM_curvature}, we see that the curvature of $\partial U_t(z,\rho_0)$ does not exceed some $\kappa_*(q)$ depending only on~$q$.
Applying Lemma~\ref{LM_krug} for $\Gamma:=\partial U_t(z,\rho_0)$ leads to the desired conclusion with $\varepsilon_0(q):=1/\kappa_*(q)$.
\end{proof}

\subsection{Proof of Theorem~\protect{\ref{TH_MAIN}}}\label{SS_main-proof}\newcommand{\STEP}[1]{\vskip1ex\noindent{\bf Step~#1.}~}
We use the notation introduced in Sect.\,\ref{SS_main-construction}\,--\,\ref{SS_dist-diam}.
In particular, for $t> t_1$ we denote by $\tilde g_t$  the conformal map of~$\Delta$, normalized as usual by $\tilde g_t(\infty)=\infty$, $\tilde g_t'(\infty)>0$, onto the unbounded component of $\ComplexE\setminus L(\Gamma_t)$, where $\Gamma_t:=\partial G(e^{-t}\UD)=\partial g_t(\Delta)$.

\STEP1 Fix some $t>t_1$. Let us obtain a lower estimate of $|\tilde g'_t(z)|$ for $z\in\partial\Delta$. Denote $\rho(t):=\tilde g_t'(\infty)$. Note that $\partial \tilde g_t(\Delta)=L(\Gamma_t)=G\big(e^{-t_0}T(\{z:|z|=e^{t_0-t}\})\big)$. Hence $\partial \tilde g_t(\Delta)$ is a $k$-quasidisk with $k:=3q$. It follows, see e.g. \cite[Lemma~8.1.1 on p.\,100]{KariHag}, that
$$
 g:={\rho(t)}^{-1}\tilde g_t\in\Sigma(k'),\quad\text{where $k':=2k/(1+k^2)$.}
$$
Therefore, we may apply Corollary~\ref{CR_Henkin} with $k'$ substituted for~$k$ and with $u(w):=\varphi\big(\rho(t)w\big)-e^{a(q)(t-t_0)}$, where $\varphi$ and $a(q)$ are defined in Proposition~\ref{PR_subharmonic}. The function $u$ is subharmonic in ${U:={\rho(t)}^{-1}D_{t_0}}$.

Since the curve $L(\Gamma_t)$ lies inside $L(\Gamma_{t_1})=\Gamma_{t_1}$, we have $\rho(t)<\rho(t_1)=e^{-t_1}<1$ and
$$
\dist(\partial U,\partial g(\Delta))={\rho(t)}^{-1}\dist(\Gamma_{t_0},L(\Gamma_t))> \dist(\Gamma_{t_0},\Gamma_{t_1})\ge M_*(q).
$$
The last inequality holds by Lemma~\ref{LM_Loewner-front-distance}. It follows that the hypothesis of Corollary~\ref{CR_Henkin} is satisfied with $\mu:=\min\{M_*(q),8\}$.

Note that $\tfrac14\diam(\tilde g_t(\Delta))\le\rho(t)\le\tfrac12\diam(\tilde g_t(\Delta))$, see e.g. \cite[\S5.2,\,\S5.3]{Ransford}. Hence by Lemma~\ref{LM_diam},
\begin{equation}\label{EQ_est-capacity}
\frac{M_3(q)}{4}e^{-t}\le\rho(t)\le\frac{M_4(q)}{2}e^{-t}.
\end{equation}

Now we can estimate the denominator in~\eqref{EQ_Henkin1}. Namely, we are going to obtain an upper bound for $|\nabla u|$ on $\partial g(\Delta)={\rho(t)}^{-1}L(\Gamma_t)$.

Note that $\varphi(w)=\phi(w)^{-a(q)}-e^{a(q)(t-t_0)}$ for all $w\in D_{t_0}$, where $\phi:=|(F\circ T)^{-1}|$ and $F(z):=G(e^{-t_0}z)$. Hence, for any $w\in \partial g(\Delta)$, we have:
\begin{multline*}
\big|\nabla u(w)\big|=\rho(t) \big|\nabla\varphi\big(\rho(t)w\big)\big|=\\ =a(q)\rho(t)\phi\big(\rho(t)w\big)^{-a(q)-1}\,\big|\nabla\phi\big(\rho(t)w\big)\big|\le a(q)\rho(t)\phi\big(\rho(t)w\big)^{-a(q)-1}\frac{|(T^{-1})'(z)|}{\mathsf D_*F(z)},
\end{multline*}
where $z:=F^{-1}\big(\rho(t)w\big)\in\UD$.

Note that by construction for any $w\in \partial g(\Delta)$, $\phi\big(\rho(t)w\big)=e^{t_0-t}$. Moreover, for all $z\in\UD$,
$$
\big|(T^{-1})'(z)\big|\le(1+|z_0|^2)\frac{1-|z_0|}{(1+|z_0|)^3}\le (1+3q)\frac{1+\sqrt{3q}}{(1-\sqrt{3q})^3}.
$$
Therefore, taking into account~\eqref{EQ_est-capacity} and~\eqref{EQ_DF}, we may conclude that there exists a constant $M_7(q)$ depending only on~$q$ such that
\begin{equation}\label{EQ_grad-u}
\big|\nabla u(w)\big|\le a(q) M_7(q) e^{a(q)(t-t_0)}\quad \text{for all~$w\in\partial g(\Delta)$.}
\end{equation}

Let us now estimate the numerator in~\eqref{EQ_Henkin1}.
According to Lemma~\ref{LM_Loewner-front-distance} and the lower estimate in~\eqref{EQ_est-capacity}, if $s\in[t_2,t)$ and
\begin{equation}\label{EQ_s-t}
e^{t-s}-1\le\frac{M_3(q)\alpha(k')\mu^{K'}}{4M_2(q)},\quad K':=\frac{1+k'}{1-k'}=K^2,
\end{equation}
where $\alpha(\cdot)$ is defined in the statement of Corollary~\ref{CR_Henkin}, then $\mathcal B_g\big(\alpha(k')\mu^{K'}\big)$ lies in the unbounded component of $\Complex\setminus{\rho(t)}^{-1}L(\Gamma_s)$. Since on ${\rho(t)}^{-1}L(\Gamma_s)$ the function~$u$ is equal identically to~$u_s:=e^{a(q)(s-t_0)}-e^{a(q)(t-t_0)}$, we conclude that under condition~\eqref{EQ_s-t} the infimum in~\eqref{EQ_Henkin1} is greater or equal to
$
-u_s=e^{a(q)(t-t_0)}\big(1-e^{-a(q)(t-s)}\big).
$
Recall that
$$
 t-t_2>t_1-t_2=\tfrac12\log\frac{2(1+|z_0|^2)}{(1+|z_0|)^2}\ge\frac12\log\frac{2(1+3q)}{(1+\sqrt{3q})^2}>0,
$$
see~\eqref{EQ_t_1}, \eqref{EQ_t_2}, and~\eqref{EQ_z_0-est}. Choosing the smallest $s\in[t_2,t)$ satisfying~\eqref{EQ_s-t} and taking into account that $a(q)$,  $\mu=\min\{M_*(q),8\}$, as well as $k$ and $k'$, depend only on~$q$, we conclude that the numerator of~\eqref{EQ_Henkin1} is bounded from below by~$M_8(q)e^{a(q)(t-t_0)}$, where $M_8(q)>0$ depends only on~$q$. Thus, taking into account~\eqref{EQ_grad-u}, from~\eqref{EQ_Henkin1} we obtain
\begin{equation}\label{EQ_STEP1}
  |\tilde g_t'(z)|\ge \frac{M_{9}(q)}{a(q)}\tilde g_t'(\infty)\ge \frac{M_{9}(q)M_{3}(q)}{4a(q)}e^{-t}\quad \text{for all $z\in\partial\Delta$ and all $t>t_1$},
\end{equation}
where $M_9(q)>0$ depends only on~$q$.

\STEP2  Now we obtain an upper estimate for $|\tilde g'_t|$ for $t>t_1$. Again consider $$g:={\rho(t)}^{-1}\tilde g_t\in\Sigma(k'),\quad\rho(t):=\tilde g'_t(\infty).$$ By Lemma~\ref{LM_tangent-disk} and the right inequality in~\eqref{EQ_est-capacity},  the function~$g$ satisfied the hypothesis of Proposition~\ref{PR_derivative-from-ABOVE} with any point~$w_0\in\partial g(\Delta)$ and with any $\varepsilon>0$ not exceeding $\rho(t)^{-1}\varepsilon_0(q)$. By~\eqref{EQ_est-capacity}, $\rho(t)^{-1}\ge 2e^t/M_4(q)>2/M_4(q)$. Therefore, we can take $\varepsilon:=2\varepsilon_0(q)/M_4(q)$. Taking into account that when $g'$ exists on the boundary, it clearly coincides with $\angle g'$,  by Proposition~\ref{PR_derivative-from-ABOVE} we have
\begin{equation}\label{EQ_STEP2}
  |\tilde g_t'(z)|\le \mathcal M(\varepsilon,k')\tilde g_t'(\infty)\le M_{10}(q)e^{-t}\quad \text{for all $z\in\partial\Delta$, all $t>t_1$},
\end{equation}
and some $M_{10}(q)$ depending only on~$q$.

\STEP3 Now we will see that $\tilde g_t$ satisfies the Loewner\,--\,Kufarev equation and find bounds for the real part of Herglotz function.

Applying  Proposition~\ref{PR_C2-homotopy} with $\zeta\mapsto1/F(T(e^{t_0}/\zeta))$, $a:=t_1<|\zeta|<b:=+\infty$, substituted for $\Psi$, we see that $\tilde g_t$ is differentiable w.r.t. $t$ for all~$t>t_1$ and that it satisfies the Loewner\,--\,Kufarev PDE in~$\Delta$,
$$
 \partial \tilde g_t(z)/\partial t = z\tilde g_t'(z)p(z,t),\quad t>t_1,~z\in\Delta,
$$
with the Herglotz fuction $p:\Delta\times[0,+\infty)\to\Complex$, $\Im p(\infty,t)=0$, determined for each fixed~$t>t_1$ by the real part its continuous extension to~$\partial\Delta$:
\begin{multline}\label{EQ_Re-Herglotz-in-Delta}
  \Re p(e^{i\theta},t)=\\=\frac{1}{|\tilde g'_t(e^{i\theta})|} \Im\left(\overline{\frac{\partial F\big(T(e^{t_0-t}e^{i\tau})\big)}{\partial \tau}}\,\frac{\partial F\big(T(e^{t_0-t}e^{i\tau})\big)}{\partial t}\right)\left|\frac{\partial F\big(T(e^{t_0-t}e^{i\tau})\big)}{\partial \tau}\right|^{-1}
\end{multline}
for all~$\theta\in[0,2\pi]$, where $\tau=\tau(\theta)$ satisfies $F\big(T(e^{t_0-t}e^{i\tau})\big)=\tilde g_t(e^{i\theta})$. In particular,
$$
  \Re p(e^{i\theta},t)\le\frac{1}{|\tilde g'_t(e^{i\theta})|}\left|\frac{\partial F\big(T(e^{t_0-t}e^{i\tau})\big)}{\partial t}\right|.
$$
It follows easily with the help of~\eqref{EQ_STEP1}, \eqref{EQ_T-prime-from-above}, and Lemma~\ref{LM_dF} that
\begin{equation}\label{EQ_STEP3}
  \Re p(e^{i\theta},t)\le M_{11}(q)\quad \text{for all~$t>t_1$ and all~$\theta\in[0,2\pi]$},
\end{equation}
where $M_{11}(q)$ is a constant depending only on~$q$.

A lower estimate for $\Re p$ is a bit more tricky. We have to use the fact that $F$ is a smooth $k$-q.c. map. In particular, if $\eta_1$ and $\eta_2$ are images of some $\xi\in\Complex$ and $i\xi$ w.r.t. the differential of~$F$, respectively, then ${\Im(\overline{\eta_1}\,\eta_2)\ge (1-k^2)(1+k^2)^{-1}|\eta_1|\,|\eta_2|}$. Taking into account that
\begin{equation}\label{EQ_Tprime-from-below}
|T'(z)|\ge(1+|z_0|^2)\frac{1-|z_0|}{(1+|z_0|)^3}\ge(1+3q)\frac{(1-\sqrt{3q})}{(1+\sqrt{3q})^3} \quad\text{for all~$z\in\UD$},
\end{equation}
with the help of~\eqref{EQ_STEP2} and Lemma~\ref{LM_dF}, we obtain:
\begin{equation}\label{EQ_STEP3a}
  \Re p(e^{i\theta})\ge \frac{1-k^2}{1+k^2}\,\frac{\,\mathsf D_* F\big(T(e^{t_0-t}e^{i\tau})\big)\,\big|e^{t_0-t}T'(e^{t_0-t}e^{i\tau})\big|\,}{|\tilde g'_t(e^{i\theta})|}\ge M_{12}(q)>0
\end{equation}
for all~$t>t_1$ and all~$\theta\in[0,2\pi]$, where $M_{12}(q)$ depends only on~$q$.

\STEP4 Now we have to find an estimate for the modulus of continuity of $\Re p$ on $\partial\Delta$. Fix $t>t_1$. Denote $z(\tau):=e^{t_0-t}e^{i\tau}$.  A simple calculation shows that
$$
\Phi(\tau):=-ie^{t-t_0}\frac{\partial F\big(T(z(\tau))\big)}{\partial \tau}=e^{i\tau}T'(z(\tau))\,\partial F\big(T(z(\tau))\big)-\overline{e^{i\tau}T'(z(\tau))}\,\bar\partial F\big(T(z(\tau))\big).
$$
From~\eqref{EQ_Re-Herglotz-in-Delta} we obtain
\begin{multline}\label{EQ_step4a}
  \log|\tilde g'_t(e^{i\theta})|+\log\Re p(e^{i\theta},t)= \\
  =\log J_F\big(T(z(\tau))\big)+t_0-t+2\Re\log T'(z(\tau))-\Re\log \Phi(\tau)=:V(\tau),
\end{multline}
where $\tau$ and $\theta$ are related, as  above, by the equality   $F\big(T(z(\tau))\big)=\tilde g_t(e^{i\theta})$. With $\tau$  regarded as a function of~$\theta$, we have $\di \tau/\di\theta=-e^{t-t_0}|\tilde g'_t(e^{i\theta})|/|\Phi(\tau)|$. Hence the l.h.s. of~\eqref{EQ_step4a} is differentiable function of~$\theta$ and
$$
\left|\frac{\partial}{\partial\theta}\Big(\log|\tilde g'_t(e^{i\theta})|+\log\Re p(e^{i\theta},t)\Big)\right|=
e^{t-t_0}\frac{|\tilde g'_t(e^{i\theta})|}{|\Phi(\tau)|}|V'(\tau)|\le\frac{e^{t-t_0}|\tilde g'_t(e^{i\theta})|}{\mathsf D_* F\big(T(z(\tau))\big)}\,\frac{|V'(\tau)|}{|T'(z(\tau))|}
$$
With the help of~\eqref{EQ_STEP2}  and Lemma~\ref{LM_dF}, it follows that
\begin{equation}\label{EQ_est1}
\left|\frac{\partial}{\partial\theta}\Big(\log|\tilde g'_t(e^{i\theta})|+\log\Re p(e^{i\theta},t)\Big)\right| \le M_{13}(q)\frac{V'(\tau)}{\big|T'(z(\tau))\big|}
\end{equation}
for some constant~$M_{13}(q)$ depending only on~$q$.

Now let us estimate $|V'(\tau)|/\big|T'(z(\tau))\big|$.
Using inequality~\eqref{EQ_est-derivative-of-Jacobian} in Lemma~\ref{LM_est-second-formal} and taking into account that $$J_F=|\partial F|^2-|\bar\partial F|^2\ge(1-k^2)|\partial F|^2,$$ we obtain
\begin{equation}\label{EQ_est2}
  \frac{1}{\big|T'(z(\tau))\big|}\,\left|\frac{\di}{\di\tau}\log J_F\big(T(z(\tau))\big)\right|\le \frac{8M(q)}{1-k^2}e^{t_0-t}.
\end{equation}
Moreover, it is straightforward to check that
\begin{equation}\label{EQ_est3}
  \frac{1}{\big|T'(z(\tau))\big|}\,\left|\frac{\di}{\di\tau}\log T'(z(\tau))\right|=e^{t_0-t}\frac{\,|T''(z(\tau))|^{\hphantom{1}}}{|T'(z(\tau))|^2}
\end{equation}
and that
\begin{equation}\label{EQ_est3a}
\frac{\,|T''(z(\tau))|^{\hphantom{1}}}{|T'(z(\tau))|^2}= 4|z_0|\,\frac{\big|1+|z_0|^2+2\bar z_0z\big|}{1-|z_0|^4}\le 4\,\frac{\sqrt{3q}(1+\sqrt{3q})}{(1-\sqrt{3q})^3}.
\end{equation}
Again by a straightforward computation,
\begin{multline*}
  \Phi'(\tau)=ie^{i\tau}T'\partial F+i\overline{e^{i\tau}T'}\,\bar\partial F~+\\
  +~e^{t-t_0}\Big(e^{2i\tau}T''\partial F-\overline{e^{2i\tau}T''}\,\bar\partial F~+~ \big(e^{i\tau}T'\big)^2\partial^2 F-\overline{\big(e^{i\tau}T'\big)^2}\,\bar\partial^2 F\Big),
\end{multline*}
where $T$ and its derivatives are calculated at the point~$z(\tau)$, while the derivatives of~$F$ are calculated at~$T(z(\tau))$.
Using Lemma~\ref{LM_est-second-formal} and bearing in mind that $|\bar\partial F|\le|\partial F|$, form the above formula we get:
\begin{equation}\label{EQ_est4pre}
 \frac{\Phi'(\tau)}{|T'|}~\le~2\,\big|\partial F\big| \Big(1+e^{t_0-t}\Big(M(q)\big|T'\big|+\big|T''\big|/\big|T'\big|\Big)\Big).
\end{equation}
Since $|\Phi(\tau)|\ge |T'|\big(|\partial F|-|\bar\partial F|\big)\ge(1-k)\,\big|T'\big|\,\big|\partial F\big|,
$
from~\eqref{EQ_est4pre} it follows that
\begin{equation}\label{EQ_est4}
 \frac{1}{|T'(z(\tau))|}\,\left|\frac{\di}{\di\tau}\log\Phi(\tau)\right|~\le~\frac{2}{1-k}\, \left(\frac{1}{|T'(z(\tau))|}+e^{t_0-t}\Big(M(q)+\frac{\,|T''(z(\tau))|^{\hphantom{1}}}{|T'(z(\tau))|^2}\Big)\right)\!.
\end{equation}

Combining \eqref{EQ_est1}\,--\,\eqref{EQ_est3a}, \eqref{EQ_est4}, and~\eqref{EQ_Tprime-from-below} and bearing in mind that~$t>t_1>t_0$, we see that there exists a constant $M_{14}(q)>0$ depending only on~$q$ such that
\begin{equation}\label{EQ_step4b}
  \left|\frac{\partial}{\partial\theta}\Big(\log|\tilde g'_t(e^{i\theta})|+\log\Re p(e^{i\theta},t)\Big)\right|\le M_{14}(q)\quad \text{for all~$t>t_1$ and~$\theta\in[0,2\pi]$.}
\end{equation}

To complete Step 4, it remains to estimate the modulus of continuity of $\theta\mapsto \log|\tilde g'_t(e^{i\theta})|$. Denote by $\omega$ the modulus of continuity of the tangent unit vector  $\beta(s)$ to $\partial\tilde g_t(\Delta)$ regarded as a function of the length parameter~$s$. Then $\omega(\delta)\le\kappa_0(q,t)\delta$, where $\kappa_0(q,t)$ is the upper bound for the curvature of $\partial\tilde g_t(\Delta)$ given in Lemma~\ref{LM_curvature}. Following the argument from~\cite{Warschawski1961}, for the modulus of continuity of $\theta\mapsto \arg\tilde g'_t(e^{i\theta})$ denoted by~$\omega_0$, we have
\begin{equation}\label{EQ_modulus-arg}
  \omega_0(\delta)\le \delta+\omega\big(\delta\max|\tilde g'_t(e^{i\theta})|\big)\le \delta\Big(1+ \kappa_0(q,t)M_{10}(q)e^{-t}\Big),
\end{equation}
where the maximum is taken over all~${\theta\in[0,2\pi]}$ and the last inequality holds because of~\eqref{EQ_STEP2}.

Notice that $\log|\tilde g'_t|$ is harmonic conjugate to $\arg\tilde g'_t$. Denote the modulus of continuity of $\theta\mapsto \log|\tilde g'_t(e^{i\theta})|$ by~$\omega_0^*$.  Using the well-known inequality due to Zygmund~\cite{Zygmund}, see e.g. \cite[Theorem~1.3  in Chap.\,III]{Garnett}, and Lemma~\ref{LM_curvature}, we obtain
\begin{eqnarray}
\nonumber
\omega_0^*(\delta)&\le&%
A\Big(\int_0^\delta\!\frac{\omega_0(x)}{x}\,\di x+\delta\int_\delta^\pi\!\frac{\omega_0(x)}{x^2}\,\di x\Big)%
\\&\le&\nonumber%
A\Big(1+ \kappa_0(q,t)M_{10}(q)e^{-t}\Big)\delta\big(1+\log(\pi/\delta)\big)%
\\\label{EQ_Zygmund}
&\le&%
A\Big(1+ \big(M_{5}(q)+M_{6}(q)\big)M_{10}(q)\Big)\delta\big(1+\log(\pi/\delta)\big), ~ 0\le\delta\le\pi,
\end{eqnarray}
where $A$ is an absolute constant.

Denote by $\omega_1$ and $\omega_2$ the continuity moduli of $\theta\mapsto\log\Re p(e^{i\theta},t)$ and  $\theta\mapsto\Re p(e^{i\theta},t)$, respectively. Then for a fixed $t>t_1$, with the help of~\eqref{EQ_STEP3}, \eqref{EQ_step4b}, and~\eqref{EQ_Zygmund},  we have
\begin{multline}\label{EQ_omega2}
\omega_2(\delta)\le\max_{\theta\in\Real}\Re p(e^{i\theta})\,\omega_1(\delta)\le\\\le M_{11}(q)\big(M_{14}(q)\delta+\omega_0^*(\delta)\big)\le M_{15}(q)\delta + M_{16}\delta\log(\pi/\delta).
\end{multline}
for some constants $M_{15}(q)$ and $M_{16}(q)$ depending only on~$q$.

\STEP5 Finally we can estimate $\Im p$. Again we fix an arbitrary $t>t_1$. According to~\eqref{EQ_omega2},  $\Re p(\cdot,t)$ is Dini-continuous on~$\partial\Delta$. Therefore, by Zygmund's result mentioned above, $\Im p(\cdot,t)$ is continuous up to the boundary and on~$\partial\Delta$ it can be written via a suitable version of the Hilbert transform. It follows that for all~$\theta\in[0,2\pi]$,
\begin{eqnarray*}
\big|\Im p(e^{i\theta},t)\big|&=&\frac1{2\pi}\left|\int_{0}^{\pi}\frac{\Re p(e^{i(\theta+x)},t)-\Re p(e^{i(\theta-x)},t)}{x}\frac{x}{\tg(x/2)}\,\di x\right|
\\[1ex]
&\le&\frac{1}{2\pi}\int_{0}^{\pi}\frac{\omega_2(2x)}{x}\frac{x}{\tg(x/2)}\,\di x
~\le~ \frac{2}{\pi}\int_0^{\pi}\frac{\omega_2(x)}{x}\,\di x,
\end{eqnarray*}
where we took into account that by Proposition~\ref{PR_C2-homotopy}, the limit $\lim_{s\to t}\big(\tilde g_s-\tilde g_t\big)/(s-t)$ is attained locally uniformly in~$\Delta\setminus\{\infty\}$ and that $\tilde g_t'(\infty)>0$ for all~$t\ge t_1$, from which it follows that $\Im p(\infty,t)=0$.

Form the above inequality and~\eqref{EQ_omega2}, we finally obtain
\begin{equation}\label{EQ_STEP5}
\big|\Im p(e^{i\theta},t)\big|\le \frac{2}{\pi}M_{15}(q)+2M_{16}(q).
\end{equation}
Recalling~\eqref{EQ_STEP3} and~\eqref{EQ_STEP3a}, we see that there exists $k_1(q)\in(0,1)$ depending only on ${q\in(0,1/3)}$ such that for any~$t>t_1$, $p(\Delta,t)\subset U\big(k_1(q)\big)$, where the disk $U(\cdot)$ is defined in Becker's Theorem~\ref{TH_Becker}.
In particular, it follows that $t\mapsto  -\log \tilde g'_t(\infty)$ is differentiable and its derivative is bounded in~$(t_1,+\infty)$ from above and below by two positive constants (depending only on~$q$). Therefore, taking into account that by construction $\tilde g_{t_1}(\Delta)=g_{t_1}$ with $\tilde g_t(\Delta)\to g_{t_1}(\Delta)$ in Carath\'eodory's sense as $t\to t_1+0$, and setting $\tilde g_t:=g_t$ for all~$t\in[0,t_1)$, we obtain a Loewner chain~$(\tilde g_t)$ in~$\Delta$.

By construction,  $F(T(0))=0\not\in\tilde g_t(\Delta)$ for all~$t\ge0$. Hence, $(f_t)$ defined by $$f_t(\zeta):=1/\tilde g_t(1/\zeta),\quad \zeta\in\UD,~t\ge0,$$ is a Loewner chain in~$\UD$. The Herglotz function of~$(f_t)$ is simply $p_\UD(\zeta,t)=p(1/\zeta,t)$ for all $t\ge0$, $t\neq t_1$, and all~$\zeta\in\UD$. It follows that $(f_t)$ satisfies Becker's condition~\eqref{EQ_Beckers-condition} with $k:=k_1(q)$ for all~$t>t_1$ and with~$k:=3q$ for all $t\in[0,t_1)$.
Finally notice that $f_0(\zeta)=1/\tilde g_0(1/\zeta)=1/g_0(1/\zeta)=f(\zeta)$. Thus, $f$ admits a Becker $k_0(q)$-q.c. extension to~$\Complex$, where $k_0(q):=\max\{k_1(q),\,3q\}$. The proof is complete. \qed

\end{document}